\documentclass[11pt, twoside]{amsart}

\title[Braided Commutative Algebras]{Braided Commutative Algebras over\\ Quantized Enveloping Algebras}

\author{Robert Laugwitz}
\address{School of Mathematical Sciences,
University of Nottingham, University Park, Nottingham, NG7 2RD, UK}
\email{robert.laugwitz@nottingham.ac.uk}
\urladdr{https://www.nottingham.ac.uk/mathematics/people/robert.laugwitz}

\author{Chelsea Walton}
\address{Department of Mathematics, The University of Illinois at Urbana-Champaign,
273 Altgeld Hall, 1409 W. Green Street 
Urbana, IL 61801-2908}
\email{notlaw@illinois.edu}
\urladdr{https://faculty.math.illinois.edu/~notlaw/}

\usepackage{mathabx}
\usepackage{import}
\usepackage{amsmath}
\usepackage{amsfonts}
\usepackage{amsthm}
\usepackage{amssymb}
\usepackage{dsfont}
\usepackage[alphabetic, initials]{amsrefs}
\usepackage[english]{babel}
\usepackage{url}
\usepackage{fancyhdr}
\usepackage{graphicx}
\usepackage[
colorlinks=true,
linkcolor=black, 
anchorcolor=black,
citecolor=black,
urlcolor=black, 
]{hyperref}
\usepackage[all]{xy}
\usepackage{geometry}
\usepackage{microtype} 

\input xy
\xyoption{all}

\newcommand{\vin}{\rotatebox[origin=c]{90}{$\in$}}

\usepackage{enumitem}

\usepackage{calrsfs}
\DeclareMathAlphabet{\cal}{OMS}{zplm}{m}{n}

\usepackage[justification=centering]{caption}

\newcommand{\leftexp}[2]{{\vphantom{#2}}^{#1}{#2}}
\newcommand{\leftexpsub}[3]{{\vphantom{#3}}^{#1}_{#2}{#3}}

\newcommand{\lYD}[1]{\leftexpsub{#1}{#1}{\mathbf{YD}}}

\newcommand{\Set}[1]{\left\lbrace #1\right\rbrace}

\newcommand{\op}[1]{\mathrm{#1}}
\newcommand{\oop}{\mathrm{op}}

\newcommand{\rdual}[1]{\leftexp{*\hspace{-2pt}}{#1}}
\newcommand{\lmod}[1]{#1\text{-}\mathbf{Mod}}

\newcommand{\lwfmod}[1]{#1\text{-}\mathbf{Mod}^{\mathrm{lfw}}}
\newcommand{\rmod}[1]{\mathbf{Mod}\text{-}#1}

\newcommand{\lcomod}[1]{#1\text{-}\mathbf{CoMod}}

\newcommand{\Cent}{\operatorname{Cent}}

\newcommand{\ev}{\operatorname{ev}}

\newcommand{\cEnd}{\cal{E}nd}

\newcommand{\Heis}{\operatorname{Heis}}

\newcommand{\ide}{\operatorname{Id}}

\newcommand{\Alg}{\mathbf{Alg}}

\newcommand{\ComAlg}{\mathbf{ComAlg}}

\newcommand{\gm}{\gamma}

\newcommand{\Vect}{\mathbf{Vect}_\Bbbk}

\newcommand{\Ug}{U_q(\mathfrak{g})}
\newcommand{\ug}{u_q(\mathfrak{g})}

\providecommand{\fr}[1]{\mathfrak{#1}}

\providecommand{\op}[1]{\operatorname{#1}}

\newcommand{\mZ}{\mathbb{Z}}

\newcommand{\mF}{\mathbb{F}}

\newcommand{\cC}{\cal{C}}
\newcommand{\cD}{\cal{D}}
\newcommand{\cB}{\cal{B}}

\newcommand{\cM}{\cal{M}}
\newcommand{\cN}{\cal{N}}

\newcommand{\cZ}{\cal{Z}}

\newcommand{\rF}{\mathrm{F}}
\newcommand{\rG}{\mathrm{G}}

\newcommand{\rR}{\mathrm{R}}
\newcommand{\rT}{\mathrm{T}}

\newtheoremstyle{defstyle}
  {0.5cm}                   
  {0.5cm}                   
  {\normalfont}           
  {}     
  {\normalfont\bfseries}  
  {:}                     
  {0.3cm}              
  {\thmname{#1}\thmnumber{ #2}\thmnote{ (#3)}}

\numberwithin{equation}{section}

\newtheorem*{rep@theorem}{\rep@title}
\newcommand{\newreptheorem}[2]{%
\newenvironment{rep#1}[1]{%
 \def\rep@title{#2 \ref{##1}}%
 \begin{rep@theorem}}%
 {\end{rep@theorem}}}
\makeatother

\newtheorem{theorem}{Theorem}[section]

\newtheorem{proposition}[theorem]{Proposition}
\newreptheorem{proposition}{Proposition}
\newtheorem{corollary}[theorem]{Corollary}
\newreptheorem{corollary}{Corollary}
\newtheorem{lemma}[theorem]{Lemma}

\newtheorem{theorem*}{Theorem}
\newreptheorem{theorem}{Theorem}

\theoremstyle{definition}
\newtheorem{definition}[theorem]{Definition}
\newtheorem{notation}[theorem]{Notation}

\theoremstyle{remark}
\newtheorem{example}[theorem]{Example}
\newtheorem{remark}[theorem]{Remark}

\newtheorem{question}[theorem]{Question}
\newtheorem{problem}[theorem]{Problem}

\makeatletter              
\let\c@equation\c@theorem  
\makeatother
\numberwithin{equation}{section}

\subjclass[2010]{Primary 18D10; Secondary 16D90, 16T05, 17B37}
\keywords{braided commutative algebra, center, module algebra, monoidal category, Morita invariant}

\begin{document}

\begin{abstract}
We produce  braided commutative algebras in braided monoidal categories by generalizing Davydov's full center construction of commutative algebras in centers of monoidal categories. Namely, we build braided commutative algebras in relative monoidal centers $\cZ_\cB(\cC)$ from algebras in $\cB$-central monoidal categories $\cC$, where $\cB$ is an arbitrary braided monoidal category; 
Davydov's (and previous works of others) take place in the special case when $\cB$ is the category of vector spaces $\Vect$ over a field $\Bbbk$. 
Since key examples of relative monoidal centers are suitable representation categories of quantized enveloping algebras, we supply braided commutative module algebras over such quantum groups.

One application of our work is that we produce Morita invariants for algebras in $\cB$-central monoidal categories. 
Moreover, for a large class of $\cB$-central monoidal categories, our braided commutative algebras  arise as a braided version of centralizer algebras. This generalizes the fact that centers of algebras in $\Vect$ serve as Morita invariants. Many examples are provided throughout.
\end{abstract}
\maketitle


\section{Introduction}\label{sec:intro}

Let  $\Bbbk$ be a field and note that all algebraic structures in this manuscript are $\Bbbk$-linear. The purpose of this work is to systematically produce and study braided commutative algebras (or, commutative algebras, for short) in a certain well-behaved class of braided monoidal categories. This is achieved by generalizing Davydov's {\it full center} construction in \cites{Dav1,Dav2} for commutative algebras in centers of monoidal categories $\cZ(\cC)$, which was built on works of Fr\"{o}hlich--Fuchs--Runkel--Schweigert \cite{FFRS} and of Kong--Runkel \cite{KR} in their studies of algebras in modular tensor categories.

\smallskip
Braided commutative algebras are interesting mathematically for several reasons. For instance, they can be used to provide natural examples of bialgebroids, which are generalizations of $\Bbbk$-bialgebras with a base algebra possibly larger than $\Bbbk$. Namely, for a bialgebra $H$ and an algebra $A$ in the (braided monoidal) category of $H$-Yetter--Drinfeld modules, we have that $A$ is commutative if and only if the smash product algebra $A\rtimes H$ admits the structure of a bialgebroid with base algebra $A$ \cite{BM}*{Theorem~4.1}, \cite{Lu}*{Theorem~5.1}. 

\smallskip

Commutative algebras in braided monoidal categories also have applications to physics. For instance,  {\it extended chiral algebras}  in rational conformal field theory (RCFT) arise as commutative algebras in modular tensor categories \cite{FRS}*{Section~5.5}. These algebras were shown to be Morita invariants in  modular tensor categories, and  are used to prove that in two-dimensional RCFT there cannot be several incompatible sets of boundary conditions for a given bulk theory \cite{KR}. 

\smallskip

Moreover, commutative algebras in braided monoidal categories $\cC$ have been used to classify certain extensions of vertex operator algebras; see \cite{HKL}
and \cite{HK} for more details.

\smallskip

We anticipate that our  construction of braided commutative algebras here will have similar and new implications both in mathematics and physics. For now, note that we  deliver a supply of commutative  algebras in (braided monoidal) representation categories of quantized enveloping algebras, a result that extends beyond work in \cites{Dav1,Dav2} as we discuss below.

\smallskip

In this work, we build commutative algebras in {\it relative monoidal centers} $\cZ_\cB(\cC)$ [Definition~\ref{def:ZBC}], which is a class of braided monoidal categories studied by the first author  \cites{L15, L18} (motivated by  \cites{BD,Maj99}, and related to \cite{Mue}*{Definition 2.6}; see also \cite{DNO}*{Section~4}). Here, $\cB$ is a braided monoidal category, and $\cC$ is a monoidal category that is  {\it $\cB$-central} [Definition~\ref{def:augment}]. For instance, when  $\cB$ is the category of $\Bbbk$-vector spaces  $\Vect$, we have that $\cZ_\cB(\cC)$ is the usual monoidal center $\cZ(\cC)$ of $\cC$ \cites{JS, Maj91}. In general, $\cZ_\cB(\cC)$ is a proper subcategory of $\cZ(\cC)$ [Proposition~\ref{prop:rel-center}, Example~\ref{propersubcat}]. Analogous to Davydov's full center construction for commutative algebras in $\cZ(\cC)$  \cite{Dav1}, we show that if there exists a functor 
\begin{equation}
\label{Rb-intro}
\rR_\cB\colon \cC \to \cZ_\cB(\cC)
\end{equation}
that is a right adjoint to the forgetful functor $\cZ_\cB(\cC) \to \cC$, then it is  lax monoidal [Lemma~\ref{lem:RB}] (so it sends algebras in $\cC$ to  algebras in $\cZ_\cB(\cC)$). Our method for producing commutative algebras in $\cZ_\cB(\cC)$  is called the {\it $\cB$-center} construction; see Section~\ref{sec:Bcenter} and Theorem~\ref{thm:Bcenter}.

\smallskip

Key examples of Davydov's work occur when $\cC=\lmod{H}$, the category of modules\footnote{Throughout the paper, all modules are left modules unless stated otherwise. } over a Hopf algebra $H$; if $H$ is finite-dimensional, then
$\cZ(\cC)$ is equivalent to the category of modules over the Drinfeld double of $H$ \cites{Dri1,Dri2}.  
Now we construct a larger class of commutative algebras in braided monoidal categories, including commutative algebras in representation categories of {\it braided Drinfeld doubles} \cite{L15} (or, of {\it double bosonizations} \cite{Maj99}), including those in our title. In particular, take $\mathfrak{g}$ a semisimple Lie algebra over $\Bbbk$ with positive/negative nilpotent part $\mathfrak{n}^{+/-}$ and positive/negative Borel part $\mathfrak{g}^{+/-}$. Then, 

\smallskip

\begin{itemize}
\item $\cZ_\cB(\cC) \simeq \lwfmod{\Ug}$, the category of locally $\mathfrak{n}^+$-finite weight modules over a quantized enveloping algebra of $\mathfrak{g}$ over $\Bbbk(q)$, for $q$ a generic variable,  when 
\item $\cC = \lmod{U_q(\mathfrak{g^-})}^{\mathrm{w}}$, the category of weight modules over the negative Borel part, and
\item $\cB = \lcomod{K}$, where $K = U_q(\mathfrak{h})$ is the quantized enveloping algebra of the Cartan subalgebra $\mathfrak{h}$ of $\mathfrak{g}$, also realized as a group algebra of a lattice;
\end{itemize}

\vspace{-.7in}

\noindent($\dag$)

\vspace{.55in}

\noindent see \cite{L18}*{Section~4.3}. One can also work with small quantum groups, cf. \cite{L18}*{Section~4.4}:

\begin{itemize}
\item $\cZ_\cB(\cC) \simeq \lmod{u_q(\mathfrak{g})}$, the category of modules over a finite-dimensional quantized enveloping algebra of $\mathfrak{g}$, for $q$ a root of unity, when 
\item $\cC = \lmod{u_q(\mathfrak{g^-})}$, and
\item $\cB = \lmod{K}$, where $K$ is a 
group algebra of a certain finite abelian group.
\end{itemize}

\vspace{-.55in}

\noindent($\ddag$)

\vspace{.4in}

Let us consider a setting more general than $(\dag, \ddag)$ as follows. Take:

\vspace{.3in}

\noindent($\star$)

\vspace{-.5in}

\begin{itemize}
\item $K$, a quasi-triangular Hopf algebra;
\item $\cB = \lmod{K}$, the braided monoidal category of $K$-modules  with a braiding $\Psi$, 
\item $H$, a braided Hopf algebra in $\cB$; and 
\item $\cC = \lmod{H}(\cB)$, the monoidal category of $H$-modules in $\cB$ [Example~\ref{ex:augment}(3)].
\end{itemize}

\noindent Then, by \cite{L18}*{Example~3.35 and Proposition~3.36} (as recorded in [Proposition~\ref{prop:YD}]):
\begin{itemize}
\item $\cZ_\cB(\cC) \simeq \lYD{H}(\cB)$, the category of $H$-Yetter--Drinfeld modules  over $\cB$ \cites{Bes,BD}.
\end{itemize}

\noindent
We also recall in Definition~\ref{def:DrinKH} and Proposition~\ref{prop:Phi} that there is a functor $\Phi$ from $\lYD{H}(\cB)$ to the category of representations of the  braided Drinfeld double  $\text{Drin}_K(H^{*}, H)$ . Here, $\text{Drin}_K(H^{*}, H)$ is the usual Drinfeld double of $H$ when $K = \Bbbk$ and $H$ is a finite-dimensional Hopf algebra. In a special case, $$\Ug \cong \text{Drin}_{U_q(\mathfrak{h})}(U_q(\mathfrak{n}^+), U_q(\mathfrak{n}^-))$$ 
as shown in \cite{Maj99}*{Section~4}, see also \cite{L17}*{Section~3.6}. 

\smallskip

We verify the functor $\rR_\cB$ from \eqref{Rb-intro} exists in setting ($\star$) [Theorem~\ref{prop:RB-YD}]. We also establish under the setting $(\star)$ that the image of an algebra $A$ in $\cC$  under $\rR_\cB$  is a braided version of the centralizer algebra $\text{Cent}_{A\rtimes H}^l(A)^{\Psi^{-1}}$ of $A$ in $A\rtimes H$
[Theorem \ref{cent-thm}]. This is analogous to the main result of \cite{Dav2} in the case $\cB = \Vect$. 

\smallskip
Our main constructions and results are summarized in Figure~1 below for setting $(\star)$ when $\cC = \lmod{H}(\cB)$, {\it although much of the work below holds for arbitrary $\cB$-central monoidal~categories}.

\begin{figure}[h]
{\small
\[
\xymatrix@R=.0pc@C=2.6pc{
\cC= \lmod{H}(\cB) \ar[rr]^{\hspace{.2in}\exists \;\rR_\cB \text{~(lax monoidal)}}_{\hspace{.2in}\text{[Theorem~\ref{prop:RB-YD}]}}&
&  \lYD{H}(\cB) \ar[r]^{\hspace{-.45in}\Phi \text{~($\sim$ if $H$ fin. dim.)}}_{\hspace{-.45in}\text{[Prop.~\ref{prop:Phi}]}} & \lmod{\text{Drin}_K(H^{*}, H)} \\\\\\\\
\Alg(\cC) \ar@{->}[rr]^{\hspace{-.25in}\Alg(\rR_\cB)} \ar@{.>}[rrdddddd]^{\hspace{-.2in} {\tiny\begin{tabular}{l}
\hspace{.3in} $\cB$-center $Z_\cB(-)$\\
\hspace{.3in} {[Def.~\ref{def:Bcenter}, Prop.~\ref{prop:ZBA}]} 
\end{tabular}}} && \Alg(\lYD{H}(\cB)) \ar@{->}[r]^{\hspace{-.4in}\Alg(\Phi)} \ar@{.>}[dddddd]^{{\scriptsize\begin{tabular}{l}
\tiny{left center $C^l(-)$}\\ \tiny{[Def.~\ref{def:leftcenter}]}\\
\tiny{[Prop.~\ref{prop:leftcenter}]}
\end{tabular}}}& 
\Alg(\lmod{\text{Drin}_K(H^{*}, H)}) \ar@{.>}[dddddd]^{\tiny{\begin{tabular}{l}
\text{induced by}\\ $\Alg(\Phi)$ and $C^l(-)$\end{tabular}}}\\
\overset{~\vin}{\framebox{$A$}} 
&&&&&&\\
 &&&&&\\
  &&&&&\\
   &&&&&\\
      &&&&&\\
&& \hspace{-1.4in} \framebox{$Z_\cB(A)$} \overset{\text{[Theorem~\ref{thm:Bcenter}]}}{=} \framebox{$C^l(\rR_\cB(A))$} \in \ComAlg(\lYD{H}(\cB)) \ar@{->}[r]^{\hspace{0.05in}\Alg(\Phi)} \ar[ddddd]_{\tiny{\text{Forgetful}}} & 
\ComAlg(\lmod{\text{Drin}_K(H^{*}, H)})\\\\\\\\
      &&&&&\\
&&\hspace{-1.15in}\underset{\tiny{\text{[Theorem \ref{cent-thm}]}}}{\framebox{$\Cent^l_{A\rtimes H}(A)^{\Psi^{-1}}$}} \in \hspace{.05in}\Alg(\cB)& 
}
\]
}
{\small \caption{Main constructions for setting ($\star$).\\ Straight arrows are functors and dotted arrows are algebra assignments. }}
\end{figure}

\smallskip
The {\it $\cB$-center}  was discussed above after \eqref{Rb-intro}.
The {\it left center}, considered initially in \cites{vO-Z,Ost,FFRS} for a given braided monoidal category $\cD$, was used in \cite{Dav1} to produce  commutative algebras in  $\cD$ from algebras in $\cD$. 

\smallskip

In comparison with \cites{Dav1, Dav2}, to achieve our constructions above we must use more involved techniques of graphical calculus sensitive to the order of crossing strands, since we work in braided monoidal categories $\cB$ more general than $\Vect$. In any case, with $\cB$-centers we are able to produce Morita invariants of algebras in $\cB$-central monoidal categories as discussed below.

\begin{reptheorem}{thm:Morita}
Take $\cC$ a $\cB$-central monoidal category, and algebras $A$, $A'$ in~$\cC$. Suppose that the categories of right modules over $A$ and over $A'$ in $\cC$ are equivalent as left $\cC$-module categories. Then, the $\cB$-centers $Z_\cB(A)$ and $Z_\cB(A')$ are isomorphic as commutative algebras in~$\cZ_{\cB}(\cC)$. 
\end{reptheorem}

In reference to setting $(\ddag)$, for instance, one can employ the theorem above to produce Morita invariants for $u_q(\mathfrak{n}^-)$-module algebras by using braided commutative $u_q(\mathfrak{g})$-module algebras.

\smallskip

In addition to Davydov's work \cites{Dav1, Dav2} and the first author's work on comodule algebras  over braided Drinfeld doubles \cite{L17}, our results have connections to several other articles in the literature. See several works on braided commutative algebras in Yetter--Drinfeld categories, including \cites{CFM, C-vO-Z, CW}. See also work of Montgomery--Schneider \cite{MS}, of Cline \cite{Cline}, and of Kinser and the second author \cite{KW}  on extending module algebras over Taft algebras to those over their Drinfeld double and over~$u_q(\mathfrak{sl}_2)$. On another related note, Etingof--Gelaki realized the representation category of a small quantum group $u_q(\mathfrak{g})$ as the monoidal center of a representation category of a certain quasi-Hopf algebra $A_q(\mathfrak{g})$~\cite{Etingof-Gelaki}.

\medskip

This paper is organized as follows. We discuss  categorical preliminaries in Section~\ref{sec:prelim}, including the left center construction;  we also introduce  {\it braided centralizer algebras} there. Next, we introduce and study the $\cB$-center construction and the functor R$_\cB$ in Section~\ref{sec:main}, and we also verify Theorems~\ref{thm:Bcenter} and~\ref{thm:Morita} and discuss the special case when $\cC = \cB$, which is a monoidal category central over itself. Then, we restrict our attention to setting $(\star)$ in Section~\ref{sec:cent} and show that algebra images under  R$_\cB$ are braided centralizer algebras.
Towards constructing examples for the material in Sections~\ref{sec:prelim}--\ref{sec:cent}, we discuss braided Drinfeld doubles and Heisenberg doubles in Section~\ref{sec:doubles}. Finally,
we provide examples of our results for braided commutative algebras in the representation categories of  the small quantum group $u_q(\mathfrak{sl}_2)$ in Section~\ref{sec:uqsl2}, and of  the Sweedler Hopf algebra $T_2(-1)$ in Section~\ref{sec:T2}. We discuss how to generalize the detailed work in Section~\ref{sec:uqsl2} to $U_q(\mathfrak{g})$ and $u_q(\mathfrak{g})$ in Section~\ref{sec:uqg}, and we end by listing several directions for further investigation there.


\medskip

\section{Categorical preliminaries}\label{sec:prelim}

In Section~\ref{sec:notation}, we first set up notation and conventions that we will use throughout this work. Next, we recall terminal objects and the comma category in Section~\ref{sec:terminal}, and then discuss  in Section~\ref{sec:leftcenter} the left and right center construction that produces commutative algebras in braided monoidal categories from algebras in such categories. Finally, we introduce and study braided versions of centralizer algebras  in Section~\ref{sec:braidcent}.

\subsection{Notation and conventions} \label{sec:notation}

All categories in this work are abelian, complete under arbitrary countable
biproducts, and enriched over the category of $\Bbbk$-vector spaces $\Vect$. The reader may wish to refer to \cite{EGNO} or \cite{TV} for further background information.

\smallskip

Throughout, $\cC = (\cC, \otimes)$ denotes a monoidal category; later in Section~\ref{sec:augment}, $\cC$ will be {\it $\cB$-central} in the sense of Definition~\ref{def:augment} for a braided monoidal category $\cB$.
The tensor unit of each monoidal category is denoted by $I$. We usually omit the associativity and unitarity isomorphisms for monoidal categories 
which is justified by MacLane's coherence theorem.

\smallskip

Unless stated otherwise, we assume that all monoidal functors $F\colon\cC \to \cC'$ are strong monoidal, i.e., there exists a natural isomorphism $$F_{X,Y}\colon F(X)\otimes F(Y) \overset{\sim}{\to} F(X\otimes Y)$$ that is compatible with the associativity constraints of $\cC$ and $\cC'$, and so that $F(I_\cC) \cong I_{\cC'}$.

\smallskip

We denote by $\Alg(\cC)$ the category of {\it algebras} $(A,m,u)$ in $\cC$; here, $A$ is an object of $\cC$ with associative multiplication $m\colon  A\otimes A \to A$ and unit $u\colon I \to A$. As usual, given an algebra $A$ in $\cC$, a \emph{left $A$-module} is a pair $(V,a_V)$ for $V$ an object in $\cC$ and $a_V\colon A\otimes V\to V$ a morphism in $\cC$ satisfying 
$a_V(m\otimes \ide_V)=a_V(\ide_A\otimes a_V)$ and $a_V(u\otimes \ide_V)=\ide_V$.
A morphism of $A$-modules $(V,a_V) \to (W, a_W)$ is a morphism $V \to W$ in $\cC$ that intertwines with $a_V$ and $a_W$. 
This way, we define the category $\lmod{A}(\cC)$ of left $A$-modules in $\cC$. Analogously, we define $\rmod{A}(\cC)$, the category of \emph{right} $A$-modules in $\cC$.

\smallskip

Moreover, we reserve $\cD$ to be  an arbitrary braided monoidal category with braiding $\Psi_\cD$, or $\Psi$ if $\cD$ is understood. In this case, we can consider {\it bialgebras} and {\it Hopf algebras} in  $\cD$, and we assume that all such Hopf algebras in our work have an invertible antipode. 

\smallskip

If $(A,m)$ is an algebra in $(\cD, \Psi)$, then $A$ is {\it braided commutative} (or, {\it commutative}) if $m\Psi = m$ (or, equivalently, if $m\Psi^{-1} = m$) as morphisms in $\cD$. We denote the full subcategory of $\Alg(\cD)$ of commutative algebras by $\ComAlg(\cD)$. The {\it braided opposites} of $(A,m)$ are $A^{\Psi}:=(A, m\Psi)$ and $A^{\Psi^{-1}}:=(A, m\Psi^{-1})$ with respect to the braiding and inverse braiding of $\cD$, respectively. If $A$ is braided commutative, then $A = A^{\Psi} = A^{\Psi^{-1}}$.

\smallskip

We use the graphical calculus as in \cite{L17}, similar to that used in \cite{Maj92}, for computations in (braided) monoidal categories $\cD$. The braiding in $\cD$ is denoted by $$\Psi=\vcenter{\hbox{
\begingroup%
  \makeatletter%
  \providecommand\color[2][]{%
    \errmessage{(Inkscape) Color is used for the text in Inkscape, but the package 'color.sty' is not loaded}%
    \renewcommand\color[2][]{}%
  }%
  \providecommand\transparent[1]{%
    \errmessage{(Inkscape) Transparency is used (non-zero) for the text in Inkscape, but the package 'transparent.sty' is not loaded}%
    \renewcommand\transparent[1]{}%
  }%
  \providecommand\rotatebox[2]{#2}%
  \ifx\svgwidth\undefined%
    \setlength{\unitlength}{8.22109334bp}%
    \ifx\svgscale\undefined%
      \relax%
    \else%
      \setlength{\unitlength}{\unitlength * \real{\svgscale}}%
    \fi%
  \else%
    \setlength{\unitlength}{\svgwidth}%
  \fi%
  \global\let\svgwidth\undefined%
  \global\let\svgscale\undefined%
  \makeatother%
  \begin{picture}(1,1.0011064)%
    \put(0,0){\includegraphics[width=\unitlength,page=1]{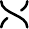}}%
  \end{picture}%
\endgroup%
}}.$$ Moreover, for a Hopf algebra $H:=(H,m,u,\Delta,\varepsilon,S)$ in $\cD$, we denote: 
\begin{gather*}
m=\vcenter{\hbox{
\begingroup%
  \makeatletter%
  \providecommand\color[2][]{%
    \errmessage{(Inkscape) Color is used for the text in Inkscape, but the package 'color.sty' is not loaded}%
    \renewcommand\color[2][]{}%
  }%
  \providecommand\transparent[1]{%
    \errmessage{(Inkscape) Transparency is used (non-zero) for the text in Inkscape, but the package 'transparent.sty' is not loaded}%
    \renewcommand\transparent[1]{}%
  }%
  \providecommand\rotatebox[2]{#2}%
  \newcommand*\fsize{\dimexpr\f@size pt\relax}%
  \newcommand*\lineheight[1]{\fontsize{\fsize}{#1\fsize}\selectfont}%
  \ifx\svgwidth\undefined%
    \setlength{\unitlength}{25.6875375bp}%
    \ifx\svgscale\undefined%
      \relax%
    \else%
      \setlength{\unitlength}{\unitlength * \real{\svgscale}}%
    \fi%
  \else%
    \setlength{\unitlength}{\svgwidth}%
  \fi%
  \global\let\svgwidth\undefined%
  \global\let\svgscale\undefined%
  \makeatother%
  \begin{picture}(1,0.59353343)%
    \lineheight{1}%
    \setlength\tabcolsep{0pt}%
    \put(0,0.59353343){\color[rgb]{0,0,0}\makebox(0,0)[lt]{\lineheight{0}\smash{\begin{tabular}[t]{l} \end{tabular}}}}%
    \put(0,0){\includegraphics[width=\unitlength,page=1]{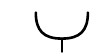}}%
  \end{picture}%
\endgroup%
}}\,, \qquad
\Delta=\vcenter{\hbox{
\begingroup%
  \makeatletter%
  \providecommand\color[2][]{%
    \errmessage{(Inkscape) Color is used for the text in Inkscape, but the package 'color.sty' is not loaded}%
    \renewcommand\color[2][]{}%
  }%
  \providecommand\transparent[1]{%
    \errmessage{(Inkscape) Transparency is used (non-zero) for the text in Inkscape, but the package 'transparent.sty' is not loaded}%
    \renewcommand\transparent[1]{}%
  }%
  \providecommand\rotatebox[2]{#2}%
  \ifx\svgwidth\undefined%
    \setlength{\unitlength}{16.80101968bp}%
    \ifx\svgscale\undefined%
      \relax%
    \else%
      \setlength{\unitlength}{\unitlength * \real{\svgscale}}%
    \fi%
  \else%
    \setlength{\unitlength}{\svgwidth}%
  \fi%
  \global\let\svgwidth\undefined%
  \global\let\svgscale\undefined%
  \makeatother%
  \begin{picture}(1,0.76387957)%
    \put(0,0){\includegraphics[width=\unitlength]{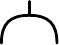}}%
  \end{picture}%
\endgroup%
}}\,, \qquad u=\vcenter{\hbox{
\begingroup%
  \makeatletter%
  \providecommand\color[2][]{%
    \errmessage{(Inkscape) Color is used for the text in Inkscape, but the package 'color.sty' is not loaded}%
    \renewcommand\color[2][]{}%
  }%
  \providecommand\transparent[1]{%
    \errmessage{(Inkscape) Transparency is used (non-zero) for the text in Inkscape, but the package 'transparent.sty' is not loaded}%
    \renewcommand\transparent[1]{}%
  }%
  \providecommand\rotatebox[2]{#2}%
  \ifx\svgwidth\undefined%
    \setlength{\unitlength}{4.8bp}%
    \ifx\svgscale\undefined%
      \relax%
    \else%
      \setlength{\unitlength}{\unitlength * \real{\svgscale}}%
    \fi%
  \else%
    \setlength{\unitlength}{\svgwidth}%
  \fi%
  \global\let\svgwidth\undefined%
  \global\let\svgscale\undefined%
  \makeatother%
  \begin{picture}(1,2.56889521)%
    \put(0,0){\includegraphics[width=\unitlength]{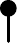}}%
  \end{picture}%
\endgroup%
}}\,, \qquad \varepsilon=\vcenter{\hbox{
\begingroup%
  \makeatletter%
  \providecommand\color[2][]{%
    \errmessage{(Inkscape) Color is used for the text in Inkscape, but the package 'color.sty' is not loaded}%
    \renewcommand\color[2][]{}%
  }%
  \providecommand\transparent[1]{%
    \errmessage{(Inkscape) Transparency is used (non-zero) for the text in Inkscape, but the package 'transparent.sty' is not loaded}%
    \renewcommand\transparent[1]{}%
  }%
  \providecommand\rotatebox[2]{#2}%
  \ifx\svgwidth\undefined%
    \setlength{\unitlength}{4.8bp}%
    \ifx\svgscale\undefined%
      \relax%
    \else%
      \setlength{\unitlength}{\unitlength * \real{\svgscale}}%
    \fi%
  \else%
    \setlength{\unitlength}{\svgwidth}%
  \fi%
  \global\let\svgwidth\undefined%
  \global\let\svgscale\undefined%
  \makeatother%
  \begin{picture}(1,2.56889521)%
    \put(0,0){\includegraphics[width=\unitlength]{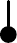}}%
  \end{picture}%
\endgroup%
}}\,, \qquad S=\vcenter{\hbox{
\begingroup%
  \makeatletter%
  \providecommand\color[2][]{%
    \errmessage{(Inkscape) Color is used for the text in Inkscape, but the package 'color.sty' is not loaded}%
    \renewcommand\color[2][]{}%
  }%
  \providecommand\transparent[1]{%
    \errmessage{(Inkscape) Transparency is used (non-zero) for the text in Inkscape, but the package 'transparent.sty' is not loaded}%
    \renewcommand\transparent[1]{}%
  }%
  \providecommand\rotatebox[2]{#2}%
  \newcommand*\fsize{\dimexpr\f@size pt\relax}%
  \newcommand*\lineheight[1]{\fontsize{\fsize}{#1\fsize}\selectfont}%
  \ifx\svgwidth\undefined%
    \setlength{\unitlength}{15.03931732bp}%
    \ifx\svgscale\undefined%
      \relax%
    \else%
      \setlength{\unitlength}{\unitlength * \real{\svgscale}}%
    \fi%
  \else%
    \setlength{\unitlength}{\svgwidth}%
  \fi%
  \global\let\svgwidth\undefined%
  \global\let\svgscale\undefined%
  \makeatother%
  \begin{picture}(1,1.04847142)%
    \lineheight{1}%
    \setlength\tabcolsep{0pt}%
    \put(0,0){\includegraphics[width=\unitlength,page=1]{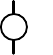}}%
    \put(-0.01126533,0.34777809){\color[rgb]{0,0,0}\makebox(0,0)[lt]{\lineheight{0}\smash{\begin{tabular}[t]{l}$+$\end{tabular}}}}%
  \end{picture}%
\endgroup%
}},\qquad S^{-1}=\vcenter{\hbox{
\begingroup%
  \makeatletter%
  \providecommand\color[2][]{%
    \errmessage{(Inkscape) Color is used for the text in Inkscape, but the package 'color.sty' is not loaded}%
    \renewcommand\color[2][]{}%
  }%
  \providecommand\transparent[1]{%
    \errmessage{(Inkscape) Transparency is used (non-zero) for the text in Inkscape, but the package 'transparent.sty' is not loaded}%
    \renewcommand\transparent[1]{}%
  }%
  \providecommand\rotatebox[2]{#2}%
  \newcommand*\fsize{\dimexpr\f@size pt\relax}%
  \newcommand*\lineheight[1]{\fontsize{\fsize}{#1\fsize}\selectfont}%
  \ifx\svgwidth\undefined%
    \setlength{\unitlength}{11.46509857bp}%
    \ifx\svgscale\undefined%
      \relax%
    \else%
      \setlength{\unitlength}{\unitlength * \real{\svgscale}}%
    \fi%
  \else%
    \setlength{\unitlength}{\svgwidth}%
  \fi%
  \global\let\svgwidth\undefined%
  \global\let\svgscale\undefined%
  \makeatother%
  \begin{picture}(1,1.37533003)%
    \lineheight{1}%
    \setlength\tabcolsep{0pt}%
    \put(0,0){\includegraphics[width=\unitlength,page=1]{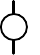}}%
    \put(-0.01477728,0.45619713){\color[rgb]{0,0,0}\makebox(0,0)[lt]{\lineheight{0}\smash{\begin{tabular}[t]{l}$-$\end{tabular}}}}%
  \end{picture}%
\endgroup%
}}.
\end{gather*}
Combining these symbols, we can display all axioms of a Hopf algebra in $\cD$ (see, e.g., \cite{L17}*{Equations~1.2--1.9}). 
For example, the bialgebra condition becomes
\begin{equation*}
\vcenter{\hbox{
\begingroup%
  \makeatletter%
  \providecommand\color[2][]{%
    \errmessage{(Inkscape) Color is used for the text in Inkscape, but the package 'color.sty' is not loaded}%
    \renewcommand\color[2][]{}%
  }%
  \providecommand\transparent[1]{%
    \errmessage{(Inkscape) Transparency is used (non-zero) for the text in Inkscape, but the package 'transparent.sty' is not loaded}%
    \renewcommand\transparent[1]{}%
  }%
  \providecommand\rotatebox[2]{#2}%
  \ifx\svgwidth\undefined%
    \setlength{\unitlength}{75.57609773bp}%
    \ifx\svgscale\undefined%
      \relax%
    \else%
      \setlength{\unitlength}{\unitlength * \real{\svgscale}}%
    \fi%
  \else%
    \setlength{\unitlength}{\svgwidth}%
  \fi%
  \global\let\svgwidth\undefined%
  \global\let\svgscale\undefined%
  \makeatother%
  \begin{picture}(1,0.40732931)%
    \put(0.31497865,0.20559355){\color[rgb]{0,0,0}\makebox(0,0)[lb]{\smash{ }}}%
    \put(0,0){\includegraphics[width=\unitlength,page=1]{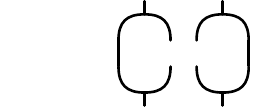}}%
    \put(0.25222581,0.17800578){\color[rgb]{0,0,0}\makebox(0,0)[lb]{\smash{$=$}}}%
    \put(0,0){\includegraphics[width=\unitlength,page=2]{bialgebra.pdf}}%
    \put(0.79876256,0.15420241){\color[rgb]{0,0,0}\makebox(0,0)[lb]{\smash{$~$}}}%
    \put(0,0){\includegraphics[width=\unitlength,page=3]{bialgebra.pdf}}%
  \end{picture}%
\endgroup%
}} \qquad \Longleftrightarrow\qquad \Delta m=(m\otimes m)(\ide\otimes \Psi\otimes \ide)(\Delta\otimes \Delta).
\end{equation*}
A left $H$-action on $V\in \cD$ is  denoted by 
$$a_V=\vcenter{\hbox{
\begingroup%
  \makeatletter%
  \providecommand\color[2][]{%
    \errmessage{(Inkscape) Color is used for the text in Inkscape, but the package 'color.sty' is not loaded}%
    \renewcommand\color[2][]{}%
  }%
  \providecommand\transparent[1]{%
    \errmessage{(Inkscape) Transparency is used (non-zero) for the text in Inkscape, but the package 'transparent.sty' is not loaded}%
    \renewcommand\transparent[1]{}%
  }%
  \providecommand\rotatebox[2]{#2}%
  \ifx\svgwidth\undefined%
    \setlength{\unitlength}{9.12051093bp}%
    \ifx\svgscale\undefined%
      \relax%
    \else%
      \setlength{\unitlength}{\unitlength * \real{\svgscale}}%
    \fi%
  \else%
    \setlength{\unitlength}{\svgwidth}%
  \fi%
  \global\let\svgwidth\undefined%
  \global\let\svgscale\undefined%
  \makeatother%
  \begin{picture}(1,1.40715315)%
    \put(0,0){\includegraphics[width=\unitlength]{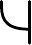}}%
    \put(0.13167899,0.48417802){\color[rgb]{0,0,0}\makebox(0,0)[lb]{\smash{$~$
}}}%
  \end{picture}%
\endgroup%
}}\colon H\otimes V\to V.$$
Moreover, the $H$-module structure $a_{V\otimes W}$ on the tensor product $V\otimes W$ of left $H$-modules $V$, $W$ becomes:
\begin{equation}\label{tensorproductaction}a_{V\otimes W}:= (a_V\otimes a_W)(\ide_H\otimes \Psi_{H,V}\otimes \ide_W)(\Delta\otimes \ide_{V\otimes W})~=~\vcenter{\hbox{
\begingroup%
  \makeatletter%
  \providecommand\color[2][]{%
    \errmessage{(Inkscape) Color is used for the text in Inkscape, but the package 'color.sty' is not loaded}%
    \renewcommand\color[2][]{}%
  }%
  \providecommand\transparent[1]{%
    \errmessage{(Inkscape) Transparency is used (non-zero) for the text in Inkscape, but the package 'transparent.sty' is not loaded}%
    \renewcommand\transparent[1]{}%
  }%
  \providecommand\rotatebox[2]{#2}%
  \ifx\svgwidth\undefined%
    \setlength{\unitlength}{30.44988496bp}%
    \ifx\svgscale\undefined%
      \relax%
    \else%
      \setlength{\unitlength}{\unitlength * \real{\svgscale}}%
    \fi%
  \else%
    \setlength{\unitlength}{\svgwidth}%
  \fi%
  \global\let\svgwidth\undefined%
  \global\let\svgscale\undefined%
  \makeatother%
  \begin{picture}(1,0.88991524)%
    \put(0,0){\includegraphics[width=\unitlength,page=1]{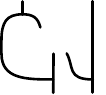}}%
    \put(0.5005321,0.47717664){\color[rgb]{0,0,0}\makebox(0,0)[lb]{\smash{$~$}}}%
    \put(0,0){\includegraphics[width=\unitlength,page=2]{tensoraction.pdf}}%
  \end{picture}%
\endgroup%
}}.
\end{equation}

Now for $H$ a Hopf algebra in a braided monoidal category $(\cD, \Psi)$, we have  via \eqref{tensorproductaction} that  $\lmod{H}(\cD)$ is a monoidal category. Similarly, $\Alg(\cD)$ is a monoidal category: The tensor product of two algebras $A,B$ in $\cD$ is given by $(A\otimes B,m_{A\otimes B}, u_{A\otimes B})$, where
\begin{align*}
m_{A\otimes B}:=(m_A\otimes m_B)(\ide_A\otimes \Psi_{A,B}\otimes\ide_B) \qquad \text{and} \qquad u_{A \otimes B}:=u_A\otimes u_B.
\end{align*}

\subsection{Terminal objects and the comma category}
\label{sec:terminal}

We record some useful results about terminal objects and the comma category that we will use below. Recall that an object $T$ of $\cC$ is {\it terminal} if, for any object $X \in \cC$, there exists a unique morphism $X \to T$ in $\cC$.
If a terminal object exists, then it is unique up to unique isomorphism. By considering the morphisms $T \otimes T \to T$ (multiplication) and $I \to T$ (unit) in $\cC$, one can verify the associativity and unit axioms to obtain the following fact.

\begin{lemma} \label{lem:terminal}
The terminal object of a (braided) monoidal category is a (commutative) algebra. \qed
\end{lemma}

Moreover, we get the following  fact from the uniqueness of morphisms to the terminal object.

\begin{lemma} \label{lem:terminal2}
If $A$ is an algebra in a monoidal category $\cC$, then the unique morphism $A\to T$ to the terminal object is a morphism of algebras in $\cC$.\qed
\end{lemma}

We will also need to use the next construction later. Take two monoidal categories $\cC$ and $\cC'$, a monoidal functor $F\colon \cC \to \cC'$, and an object $A \in \cC'$. Then, the {\it comma category} $F \hspace{-.05in} \downarrow \hspace{-.05in} A$ is the category of pairs $(X,x)$ with $X \in \cC$ and $x\colon F(X) \to A$ in $\cC'$, with morphisms being morphisms of the first components that are compatible with the morphisms of the second components. If $F$ is monoidal and $A \in \Alg{(\cC')}$, then $F \hspace{-.05in} \downarrow \hspace{-.05in} A$ is also monoidal with $$(X,x) \otimes (Y,y) := (X \otimes Y, ~m_A (x\otimes y) F^{-1}_{X,Y}\colon F(X \otimes Y) \to A).$$

\subsection{Left and right center} \label{sec:leftcenter}

Next, we describe how to associate to a given algebra $A$ in $\cD$ certain braided commutative algebras in $\cD$, analogous to the center. The definition below appears in \cite{Dav1}*{Section 5}, following (and equivalent to) \cite{Ost}*{Definition 5.1}; see also, \cite{vO-Z}*{Definition~4.3}.

\begin{definition} \label{def:leftcenter}
Let $(A,m)$ be an algebra in $(\cD,\Psi)$. The \emph{left center} $C^l(A)\to A$ of $A$ is the terminal object in the category of morphisms $\gamma\colon Z\to A$ such that  the following  diagram commutes.
\begin{align}\label{Cleftdiag}
\xymatrix@R=1.5pc@C=3pc{
Z\otimes A\ar[rr]^-{\Psi_{Z,A}}\ar[d]_{\gamma\otimes \ide}&&A\otimes Z\ar[d]^{\ide\otimes \gamma}\\
A\otimes A\ar[rd]_m&& A\otimes A\ar[ld]^m\\
&A&
}
\end{align}
Equivalently, it is defined as the maximal subobject $C^l(A)$ of $A$ such that $m\Psi_{C^l(A), A} = m$ as maps from $C^l(A) \otimes A$ to $A$ in $\cD$. 

Similarly, we define the \emph{right center} $C^r(A)\to A$ using $\Psi^{-1}$ instead of $\Psi$.
\end{definition}

\begin{proposition} \label{prop:leftcenter} \cite{Dav1}*{Proposition 5.1} The left center $C^l(A)$ has an algebra structure in $\cD$, unique up to unique isomorphism of algebras, such that $C^l(A)\to A$ is a morphism in $\Alg(\cD)$. In addition, $C^l(A) \in \ComAlg(\cD)$.  Similarly, $C^r(A) \in \ComAlg(\cD)$. \qed
\end{proposition}

\subsection{Left and right centralizers} \label{sec:braidcent}

We generalize the left and right center  constructions in the previous section as follows. 

\begin{definition} \label{def:leftcentzer}
Let $(A,m)$ be an algebra, $S$ be an object, and $\phi\colon S\to A$ be a morphism  in $(\cD, \Psi)$. The \emph{left  centralizer} $\Cent^l_A(S)\to A$ of $A$ is a terminal object in the category of morphisms  $\gamma \colon C\to A$ such that the following diagram commutes.
\begin{align}\label{diag:leftcentzer}
\xymatrix@R=1.5pc@C=3pc{
C\otimes S\ar[rr]^-{\Psi_{C,S}}\ar[d]_{\gamma\otimes \phi}&&S\otimes C\ar[d]^{\phi\otimes \gamma}\\
A\otimes A\ar[rd]_m&& A\otimes A\ar[ld]^m\\
&A&
}
\end{align}
Similarly, we define the \emph{right centralizer} $\Cent^r_A(S)\to A$ using $\Psi^{-1}$ instead of $\Psi$. 
\end{definition}

\begin{example} \label{ex:ClA} Let $A$ be an algebra in $\cD$.
\begin{enumerate}
\item Let $S=A$ and $\phi=\ide_A$. In this case, $\Cent^l_A(A)=C^l(A)$, and $\Cent^r_A(A)=C^r(A)$.

\smallskip

\item Let $S=I$, and let $\phi=u_A\colon I\to A$ be the unit of $A$. In this case, $\Cent^l_A(I)=\Cent^r_A(I)=A$.
\end{enumerate}
\end{example}

\begin{definition} \label{def:scriptC}
Let $A$ be an algebra, $S$ be an object, and $\phi\colon S\to A$ be a morphism in $\cD$. We denote by $\mathcal{C}^l_A(S)$  the category consisting of
\begin{itemize}
\item objects which are pairs $(C,\gamma)$, where $C$ is an object and $\gamma\colon C\to A$ is a morphism in $\cD$ that make Diagram~\ref{diag:leftcentzer} commute; with
\item morphisms $(C,\gamma)\to (C',\gamma')$ that are morphisms $f\colon C \to C'$ in $\cD$ such that the diagram below commutes.
\begin{align}
\xymatrix{C\ar^f[rr] \ar_{\gamma}[dr]&&C'\ar^{\gamma'}[dl]\\
&A&
}
\end{align}
\end{itemize}
From Example~\ref{ex:ClA}(1), we denote $\mathcal{C}^l_A(A)$ by $\mathcal{C}^l(A)$.
\end{definition}

Now  the left centralizer $\Cent^l_A(S)$ is the terminal object of $\mathcal{C}^l_A(S)$, and the left center $C^l(A)$ is the terminal object of $\mathcal{C}^l(A)$. Similar to Proposition~\ref{prop:leftcenter}, we have the result below.

\begin{proposition} \label{prop:leftcentzer}
The left centralizer $\Cent^l_A(S)$ has the structure of an algebra in $\cD$, which is  unique up to (unique) algebra isomorphism, such that  $\Cent^l_A(S)\to A$ is morphism in $\Alg(\cD)$. 
\end{proposition}

\begin{proof} 
By the discussion above and the material in Section~\ref{sec:terminal}, it suffices to show that $\mathcal{C}^l_A(S)$ is a monoidal category. Towards this, take $(C_1, \gamma_1)$ and $(C_2, \gamma_2)$ in $\mathcal{C}^l_A(S)$ and define
$$(C_1, \gamma_1) \otimes (C_2, \gamma_2) := (C_1 \otimes C_2, \; m_A(\gamma_1\otimes \gamma_2)).$$ Similar to \cite{Dav2}*{Remark~4.2}, this  definition satisfies \eqref{diag:leftcentzer} via the commutative diagram below. 

\medskip

\begin{align*} \label{eq:ClSAmon}
\small{
\hspace{-.1in}
\xymatrix@C=.75pc@R=1.4pc{
C_1 \otimes C_2 \otimes S 
\ar@/^1.7pc/[rrrrrr]^{\Psi_{C_1 \otimes C_2, S}}
\ar[rrr]^{\text{Id} \otimes \Psi_{C_2, S}}
\ar[rd]^{\gamma_1 \otimes \text{Id} \otimes \text{Id}}
\ar[dd]_{\gamma_1 \otimes \gamma_2 \otimes \phi}&&&
C_1 \otimes S \otimes C_2 
\ar[rrr]^{\Psi_{C_1, S} \otimes \text{Id}}
\ar[ld]_{\gamma_1 \otimes \text{Id} \otimes \text{Id}}
\ar[rd]^{\text{Id} \otimes \text{Id} \otimes \gamma_2}&&&
S \otimes C_1 \otimes C_2
\ar[ld]_{\text{Id} \otimes \text{Id} \otimes \gamma_2}
\ar[dd]^{\phi \otimes \gamma_1 \otimes \gamma_2}\\
&A \otimes C_2 \otimes S
\ar[r]^{\hspace{.1in}\text{Id} \otimes \Psi}
\ar[dl]^{\text{Id} \otimes \gamma_2 \otimes \phi}
&
A \otimes S \otimes C_2 
\ar[dr]_{\text{Id} \otimes \phi \otimes \gamma_2}&&
C_1 \otimes S \otimes A 
\ar[r]^{  \Psi \otimes \text{Id}}
\ar[dl]^{\gamma_1 \otimes \phi \otimes \text{Id}}&
S \otimes C_1 \otimes A
\ar[dr]_{\phi \otimes \gamma_1 \otimes \text{Id}}&\\
A \otimes A \otimes A 
\ar[dd]_{m \otimes \text{Id}}
\ar[dr]_{\text{Id} \otimes  m}
&&&
A \otimes A \otimes A
\ar[drr]_{m \otimes \text{Id}}
\ar[dll]^{\text{Id}  \otimes  m}&&&
A \otimes A \otimes A
\ar[dd]^{\text{Id} \otimes m }
\ar[dl]^{m \otimes \text{Id}}\\
&A \otimes A 
\ar[ddrr]_{m}&&&&
A \otimes A 
\ar[ddll]^{m}&\\
A \otimes A 
\ar[drrr]_{m} &&&&&&
A \otimes A
\ar[dlll]^{m}\\
&&&A&&&
}}
\end{align*}
\vspace{-.25in}

\end{proof}

Unlike Proposition~\ref{prop:leftcenter} (for $S=A$), neither Cent$_A^l(S)$ or Cent$_A^r(S)$ necessarily belongs to $\ComAlg(\cD)$: To see this use Example~\ref{ex:ClA}(2). In any case, consider the example below.

\begin{example} \label{ex:CentlA}
Suppose that $\cD$ is the category of modules over a quasi-triangular Hopf algebra (so it comes equipped with a fiber functor to $\Vect$ and objects have elements). 
Then, for any subobject $S\subset A$, the left centralizer $\Cent^l_A(S)$ is isomorphic to  the following subalgebra of $A$:
\begin{align*}
\Cent^l_A(S)=\Set{a\in A~\middle|~ m(a\otimes s)=m\Psi(a\otimes s),~~ \forall s\in S}.
\end{align*}
\end{example}


\medskip

\section{The functor \texorpdfstring{$\rR_\cB$}{RB} and the \texorpdfstring{$\cB$}{B}-center}\label{sec:main}
We present the main results of our work in this section. First, we discuss background material on $\cB$-central monoidal categories and relative monoidal centers $\cZ_\cB(\cC)$ adapted from  \cite{L18} in Section~\ref{sec:augment}.  In Section~\ref{sec:RB}, we study the right adjoint $\rR_\cB$ to the forgetful functor $\cZ_\cB(\cC) \to \cC$, and show that it exists when $\cZ_\cB(\cC)$ is a Yetter--Drinfeld category $\lYD{H}(\cB)$ in $\cB$ [Theorem~\ref{prop:RB-YD}]; the lax monoidal property of $\rR_\cB$ is also discussed.  Next, in Section~\ref{sec:Bcenter}, we generalize Davydov's full center construction \cite{Dav1}*{Section 4} to the $\cB$-central setting, thus producing braided commutative algebras $Z_\cB(A)$ in $\cZ_\cB(\cC)$ from algebras $A$ in~$\cC$ [Proposition~\ref{prop:ZBA}]. In Section~\ref{sec:Bctr-lfctr},  we show that  for $A \in \Alg(\cC)$ we have an isomorphism $Z_\cB(A) \cong C^l(\rR_\cB(A))$ of algebras in $\ComAlg(\cZ_\cB(\cC))$ if $\rR_\cB$ exists and its adjunction counit is an epimorphism [Theorem~\ref{thm:Bcenter}]; this generalizes \cite{Dav1}*{Theorem 5.4}. In Section~\ref{sec:Morita}, we establish that $\cB$-centers serve as Morita invariants for algebras in $\cB$-central monoidal categories [Theorem~\ref{thm:Morita}]. Finally, in Section~\ref{sec:Cbraided}, we restrict our attention to the case when $\cC$ is braided and central over itself and present the results of the previous subsections in this setting.

\subsection{\texorpdfstring{$\cB$}{B}-central monoidal categories and relative monoidal centers} \label{sec:augment}

Recall $\cB:=(\cB, \Psi)$ denotes a braided monoidal category. From now on $\cC$ will be  a monoidal category of the kind~below.

\begin{definition} \label{def:augment}
We say that a monoidal category $\cC$ is {\it $\cB$-central monoidal} if it comes equipped with a faithful monoidal functor, called the \emph{central functor},
$$\rT\colon \cB \to \cC 
$$ 
and a natural isomorphism
$$
\sigma\colon \otimes_\cC (\text{Id}_\cC \boxtimes \rT) \overset{\sim}{\implies} \otimes_\cC^{\oop} (\text{Id}_\cC \boxtimes \rT)$$
such that the assignment $B\mapsto (\rT(B),\sigma^{-1}_{B,-})$  defines a braided monoidal functor
$\overline{\cB}\to \cZ(\cC),$ where $\overline{\cB}=(\cB,\otimes,\Psi^{-1})$, i.e. the monoidal category $\cB$ with inverse braiding.
\end{definition}
The definition above generalizes the concept of a $\cB$-augmented monoidal category of \cite{L18}*{Section~3.3} relaxing the condition that the functor $\rT$ has a left inverse. See also \cite{DMNO}*{Definition~2.4} for the concept of a central functor.

\begin{example} \label{ex:augment}  \cite{L18}*{Section~3.3} Below are examples of $\cB$-central monoidal  categories.
\begin{enumerate}
\item By  the assumptions in Section~\ref{sec:notation}, all monoidal categories in this work are $\Vect$-central.

\smallskip

\item We have that $\cB$ is $\cB$-central with $\rT = \text{Id}_\cB$ and 
$\sigma = \Psi$. 
\smallskip

\item Take $H$ a Hopf algebra in $\cB$, and consider the category $\cC := \lmod{H}(\cB)$ of left $H$-modules in $\cB$. This category is monoidal: Take $(V, a_V\colon H \otimes V \to V)$ and $(W, a_W\colon H \otimes W \to W)$ in $\cC$ and we get that $(V \otimes W, a_{V \otimes W}) \in \cC$ with 
$a_{V \otimes W}$ defined as in Equation~\ref{tensorproductaction}.
Moreover, $\cC$ is $\cB$-central monoidal 
with $\rT$ giving an object of $\cB$ the structure of a trivial $H$-module in $\cB$ (via the counit of $H$), and 
$\sigma:= \rT(\Psi_{\rF(V),B})$ for all $V \in \cC$ and $B \in \cB$.
\end{enumerate}
\end{example}
 
Example~\ref{ex:augment}(3) will play a crucial role in our work later. Next, we define the relative monoidal center of a $\cB$-central monoidal category $\cC$.

\begin{definition} \cite{L18}*{Definition~3.32, Propositions~3.33 and~3.34} \label{def:ZBC}
The {\it relative monoidal center} $\cZ_\cB(\cC)$ of a $\cB$-central monoidal category $\cC$ is a braided monoidal category consisting of pairs $(V,c)$, where $V$ is an object of $\cC$, and $c := c_{V,-}\colon V \otimes \text{Id}_\cC \overset{\sim}{\Rightarrow} \text{Id}_\cC \otimes V$  is a natural isomorphism of half-braidings satisfying the two conditions below:
\begin{enumerate}
\item[(i)] [tensor product compatibility] for any $X, Y \in \cC$ the following diagram commutes: 
\[
\xymatrix{
V \otimes X \otimes Y \ar[rr]^{c_{V,X\otimes Y}} 
\ar[dr]_{c_{V,X}\otimes \text{Id}_Y}
&& X \otimes Y \otimes V
\\
&X \otimes V \otimes Y\ar[ur]_{~\text{Id}_X \otimes c_{V,Y}}&}
\]
\smallskip
\item[(ii)] [compatibility with the central functor] for any $B \in \cB$ we have $$c_{V,\rT(B)}  = \sigma_{V,B}.$$
\end{enumerate}
A morphism from $(V,c_{V,-})$ to $(W,c_{W,-})$ is defined to be a morphism $f\colon V \to W$ in $\cC$ so that for each $X \in \cC$ we have $$(\text{Id}_X \otimes f) c_{V,X} = c_{W,X}  (f \otimes \text{Id}_X).$$
Here, the monoidal structure is given by 
$$
(V,c_{V,-}) \otimes (W,c_{W,-}):= (V \otimes W, ~\{c_{V \otimes W,X} := (c_{V,X}\otimes\text{Id}_W)(\text{Id}_V \otimes c_{W,X})\}_{X\in \cC}),
$$
and the braiding is given by 
$$\Psi_{(V,c_{V,-}),(W,c_{W,-})} = c_{V,W}.$$
\end{definition}

\begin{remark}
The relative center $\cZ_\cB(\cC)$ is the \emph{M\"uger centralizer} \cite{Mue}*{Definition 2.6} of the set of objects $(\rT(B),
\sigma^{-1}_{B,V})$ in $\cZ(\cC)$.

In \cite{L18}, the relative monoidal center $\cZ_\cB(\cC)$ is defined (equivalently) as the monoidal category of $\cB$-balanced endofunctors $\rG$ of the regular $\cC$-bimodule category $\cC$; here, using composition as the tensor product, i.e., $\rG \otimes \rG' = \rG' \rG$, this is also a braided monoidal category.
\end{remark}

Relative monoidal categories have the following properties, some of which hold by definition. Here, we use \cite{L18}*{Proposition~3.34} to employ other results from \cite{L18}.

\begin{proposition} \label{prop:rel-center} 
Let $\cC$ be a $\cB$-central monoidal category over  a braided monoidal category $\cB$.
\begin{enumerate}
\item[\textnormal{(1)}] If $\cC = \cB$, then $\cZ_\cB(\cB)$ is isomorphic to $\cB$ as braided monoidal categories.

\smallskip

\item[\textnormal{(2)}] \cite{L18}*{Example~3.30} If $\cB = \Vect$, then $\cZ_\cB(\cC)$ is isomorphic to the (ordinary) monoidal center $\cZ(\cC)$ (e.g., as in \cite{EGNO}*{Definition~7.13.1}).

\smallskip

\item[\textnormal{(3)}] For $\cB$ arbitrary, $\cZ_\cB(\cC)$ is a full braided monoidal subcategory of $\cZ(\cC)$.

\smallskip

\item[\textnormal{(4)}] \cite{L18}*{Theorem~3.29} If $\cC$ is rigid (or pivotal), then so is $\cZ_\cB(\cC)$.  \qed

\smallskip

\end{enumerate}
\end{proposition}

Note that the forgetful functor $\cZ_\cB(\cC) \to \cC$ does not necessarily have a right adjoint, but we show later in Lemma~\ref{lem:RB} that if a right adjoint exists, then it is lax monoidal. In particular, such a right adjoint exists for the $\cB$-central monoidal category $\lmod{H}(\cB)$ from Example~\ref{ex:augment}(3); see Theorem~\ref{prop:RB-YD}. Toward this result, consider the following explicit description of the relative monoidal center of $\lmod{H}(\cB)$ in terms of \emph{crossed} or \emph{Yetter--Drinfeld modules}.

\begin{definition}  \label{def:HYD(B)} \cites{Bes,BD}
Take a Hopf algebra $H$ in $\cB$, and take $\cC:=\lmod{H}(\cB)$ from Example~\ref{ex:augment}(3) with left $H$-action in $\cB$ denoted by $a:=a_V:=a_V^H$. Then the category $\lYD{H}(\cB)$ of {\it $H$-Yetter--Drinfeld modules in $\cB$} consists of objects $(V,a,\delta)$ where $(V,a) \in \cC$ with left $H$-coaction  in $\cB$ denoted by $\delta:=\delta_V:=\delta_V^H$, subject to compatibility condition:
\begin{align}\begin{split} \label{eq:HYDB}
&(m_H \otimes a_V)(\text{Id}_H \otimes \Psi_{H,H} \otimes \text{Id}_V)(\Delta_H \otimes \delta_V)\\
&\quad = (m_H \otimes \text{Id}_V)(\text{Id}_H \otimes \Psi_{V,H})(\delta_V \otimes \text{Id}_H) (a_V \otimes \text{Id}_H) (\text{Id}_H \otimes \Psi_{H,V})(\Delta_H \otimes \text{Id}_V).
\end{split}
\end{align}
A morphism $f\colon (V,a_V,\delta_V)\to (W,a_W,\delta_W)$ in $\lYD{H}(\cB)$ is given by a morphism $f\colon V\to W$ in $\cB$ that is a morphism of $H$-modules and of $H$-comodules. 

Given two objects $(V,a_V,\delta_V)$ and $(W,a_W,\delta_W)$ in $\lYD{H}(\cB)$, their tensor product  is given by $(V\otimes W, a_{V\otimes W}, \delta_{V\otimes W})$, where
 $a_{V\otimes W}$  is  as in Equation~\ref{tensorproductaction} and 
\begin{align*}
\delta_{V\otimes W}=(m_H\otimes \ide_{V\otimes W})(\ide_H\otimes \Psi_{H,V}\otimes \ide_W)(\delta_V\otimes \delta_W)=\vcenter{\hbox{
\begingroup%
  \makeatletter%
  \providecommand\color[2][]{%
    \errmessage{(Inkscape) Color is used for the text in Inkscape, but the package 'color.sty' is not loaded}%
    \renewcommand\color[2][]{}%
  }%
  \providecommand\transparent[1]{%
    \errmessage{(Inkscape) Transparency is used (non-zero) for the text in Inkscape, but the package 'transparent.sty' is not loaded}%
    \renewcommand\transparent[1]{}%
  }%
  \providecommand\rotatebox[2]{#2}%
  \ifx\svgwidth\undefined%
    \setlength{\unitlength}{30.44988496bp}%
    \ifx\svgscale\undefined%
      \relax%
    \else%
      \setlength{\unitlength}{\unitlength * \real{\svgscale}}%
    \fi%
  \else%
    \setlength{\unitlength}{\svgwidth}%
  \fi%
  \global\let\svgwidth\undefined%
  \global\let\svgscale\undefined%
  \makeatother%
  \begin{picture}(1,0.88991524)%
    \put(0,0){\includegraphics[width=\unitlength,page=1]{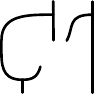}}%
    \put(0.5005321,0.41273921){\color[rgb]{0,0,0}\makebox(0,0)[lb]{\smash{$~$}}}%
    \put(0,0){\includegraphics[width=\unitlength,page=2]{tensorcoaction.pdf}}%
  \end{picture}%
\endgroup%
}}.
\end{align*}

Here, the braiding of $\lYD{H}(\cB)$ is given by
\begin{align} \label{eq:YDbraid}
\Psi^{\bf YD}_{V,W} = (a_W \otimes \text{Id}_V)(\text{Id}_H \otimes \Psi^{\cB}_{V,W})(\delta_V \otimes \text{Id}_W)= \vcenter{\hbox{
\begingroup%
  \makeatletter%
  \providecommand\color[2][]{%
    \errmessage{(Inkscape) Color is used for the text in Inkscape, but the package 'color.sty' is not loaded}%
    \renewcommand\color[2][]{}%
  }%
  \providecommand\transparent[1]{%
    \errmessage{(Inkscape) Transparency is used (non-zero) for the text in Inkscape, but the package 'transparent.sty' is not loaded}%
    \renewcommand\transparent[1]{}%
  }%
  \providecommand\rotatebox[2]{#2}%
  \ifx\svgwidth\undefined%
    \setlength{\unitlength}{32.00101968bp}%
    \ifx\svgscale\undefined%
      \relax%
    \else%
      \setlength{\unitlength}{\unitlength * \real{\svgscale}}%
    \fi%
  \else%
    \setlength{\unitlength}{\svgwidth}%
  \fi%
  \global\let\svgwidth\undefined%
  \global\let\svgscale\undefined%
  \makeatother%
  \begin{picture}(1,0.77737724)%
    \put(0,0){\includegraphics[width=\unitlength]{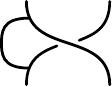}}%
  \end{picture}%
\endgroup%
}},
\end{align}
with inverse given by
\begin{align} \label{eq:YDbraidinv}
(\Psi^{\bf YD}_{V,W})^{-1} = (\ide_W\otimes (a_V\Psi_{V,H}^{-1}))(\Psi^{-1}_{V,W}\otimes S^{-1})(\ide_V\otimes (\Psi^{-1}_{H,V}\delta_W))=\vcenter{\hbox{
\begingroup%
  \makeatletter%
  \providecommand\color[2][]{%
    \errmessage{(Inkscape) Color is used for the text in Inkscape, but the package 'color.sty' is not loaded}%
    \renewcommand\color[2][]{}%
  }%
  \providecommand\transparent[1]{%
    \errmessage{(Inkscape) Transparency is used (non-zero) for the text in Inkscape, but the package 'transparent.sty' is not loaded}%
    \renewcommand\transparent[1]{}%
  }%
  \providecommand\rotatebox[2]{#2}%
  \newcommand*\fsize{\dimexpr\f@size pt\relax}%
  \newcommand*\lineheight[1]{\fontsize{\fsize}{#1\fsize}\selectfont}%
  \ifx\svgwidth\undefined%
    \setlength{\unitlength}{33.12888489bp}%
    \ifx\svgscale\undefined%
      \relax%
    \else%
      \setlength{\unitlength}{\unitlength * \real{\svgscale}}%
    \fi%
  \else%
    \setlength{\unitlength}{\svgwidth}%
  \fi%
  \global\let\svgwidth\undefined%
  \global\let\svgscale\undefined%
  \makeatother%
  \begin{picture}(1,0.92853411)%
    \lineheight{1}%
    \setlength\tabcolsep{0pt}%
    \put(0,0){\includegraphics[width=\unitlength,page=1]{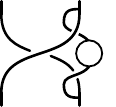}}%
    \put(0.64881035,0.38124942){\color[rgb]{0,0,0}\makebox(0,0)[lt]{\lineheight{0}\smash{\begin{tabular}[t]{l}$-$\end{tabular}}}}%
  \end{picture}%
\endgroup%
}},
\end{align}
cf. \cite{Bes}*{Theorem 3.4.3}. 
\end{definition}

\begin{proposition} \cite{L18}*{Proposition~3.36} \label{prop:YD}
For $H$ a Hopf algebra in $\cB$, consider the $\cB$-central monoidal category $\cC:= \lmod{H}(\cB)$  from Example~\ref{ex:augment}(3). Then, there is an equivalence of braided monoidal categories
$$\cZ_\cB(\cC) \overset{\sim}{\rightarrow} \lYD{H}(\cB),$$
where $(V,c) \mapsto V$ with  left $H$-module structure $a_V:= a_{V,c}\colon H \otimes V \to V$ from $\cC$, 
and $H$-coaction given by
$\delta_V := \delta_{V,c}:= c_{V,H} (\textnormal{Id}_V \otimes u_H)\colon V \to H \otimes V.$
\qed
\end{proposition}

Next, we illustrate Proposition~\ref{prop:rel-center}(3) for $\cC:= \lmod{H}(\cB)$ on the level of objects below.

\begin{example}\label{propersubcat}
Let $K$ be a quasi-triangular Hopf algebra  in $\Vect$ with $\cal{R}$-matrix $R^{(1)} \otimes R^{(2)}$, and let $H$ be a Hopf algebra in $\cB=\lmod{K}$. Then the smash product algebra $H\rtimes K$ is a Hopf algebra in $\Vect$ (called \emph{bosonization} or \emph{Radford biproduct}) such that 
$$\cC:=\lmod{H}(\cB)\simeq \lmod{H\rtimes K},$$ 
see \cite{Majid}*{Theorem 9.4.12}.
In this case, $$\cZ(\cC)\simeq \lYD{H\rtimes K}(\Vect)=:\lYD{H\rtimes K}.$$ On the other hand, by Proposition~\ref{prop:YD},
$$\cZ_\cB(\cC) \simeq \lYD{H}(\cB).$$ There is a functor from $\lYD{H}(\cB)$ to $\lYD{H\rtimes K}$, where the object
$(V,a^H_V, a^K_V,\delta^H_V) \in \lYD{H}(\cB)$ with coaction $\delta^H_V(v)=v_{(-1)}\otimes v_{(0)}$ is mapped to the object
$(V,~a^{H\rtimes K}_V, ~\delta^{H\rtimes K}_V) \in \lYD{H\rtimes K}$ with
$$a^{H\rtimes K}_V :=a^H_V(\text{Id}_H\otimes a^K_V) \qquad \text{and} \qquad
\delta^{H\rtimes K}_V(v)=v_{(-1)}\otimes R^{(2)}\otimes a_V^K(R^{(1)} \otimes v_{(0)}).$$
Hence, the essential image of this functor consists precisely of objects of $\lYD{H\rtimes K}$ isomorphic to those where the coaction $\delta$ restricted to $K$ has the form $\delta(v)=R^{(2)}\otimes a_{H\rtimes K}(R^{(1)} \otimes v)$, i.e. is induced from the action $a_{H\rtimes K}$ by using the universal $R$-matrix. This illustrates how a relative monoidal center $\cZ_\cB(\cC)$ is a proper subcategory of the monoidal center $\cZ(\cC)$ when $\cB$ is inequivalent to $\Vect$.
\end{example}

\smallskip

\subsection{A right adjoint to the forgetful  functor}\label{sec:RB} 

Let $\cC$ be a $\cB$-central monoidal category. Consider the forgetful functor $$\rF_\cB\colon \cZ_\cB(\cC)\to \cC.$$ In this section, we consider a general situation in which the forgetful functor $\rF_\cB$ has a right adjoint $\rR_\cB$. In this case, we label the corresponding adjunction natural isomorphisms as follows:
$$\alpha_W\colon W \to \rR_\cB \rF_\cB(W) \quad \text{and} \quad \beta_V\colon \rF_\cB \rR_\cB(V) \to V, \quad
W \in \cZ_\cB(\cC),~~ V \in \cC.$$

\begin{lemma} \label{lem:RB}
Assume  that $\rF_\cB \colon \cZ_\cB(\cC)\to \cC$ has a right adjoint $\rR_\cB$. Then the adjunction natural transformations $\alpha$ and $\beta$ are monoidal, and $\rR_\cB$ is lax monoidal.
\end{lemma}
\begin{proof}
This follows from a general fact in category theory. The functor $\rF_\cB$ is strong monoidal, so it is oplax monoidal, and hence its right adjoint $\rR_{\cB}$ is lax monoidal \cite{Kel} (see also \cite{nlab}). A direct proof of these results is also given in \cite{Dav1}*{Section~5}.
\end{proof}

Now assume that $H$ is Hopf algebra in $\cB$, that $\cC=\lmod{H}(\cB)$, and recall from Proposition~\ref{prop:YD} that $\cZ_{\cB}(\cC)\simeq \lYD{H}(\cB)$. We construct a right adjoint $\rR_\cB$ to $\rF_\cB$ in the next result; this generalizes the construction from \cite{CMZ}*{Corollary 2.8}  when $\cB=\Vect$ (which is used crucially  in~\cite{Dav2}).

\begin{theorem} \label{prop:RB-YD}
The forgetful functor $\rF_\cB \colon \lYD{H}(\cB)\to \lmod{H}(\cB)$ has a lax monoidal right adjoint $\rR_{\cB}\colon \lmod{H}(\cB)\to \lYD{H}(\cB)$. Moreover, the functor $\rR_\cB$ sends an object $(V,a_V)$ to the object $(H\otimes V, a^{\rR},\delta^{\rR})$ with $H$-action $a^{\rR}$ and $H$-coaction $\delta^{\rR}$ given by
\begin{align*}
a^{\rR}&=(m\otimes\ide_V)(\ide_H \otimes \Psi_{V,H})(\ide_H \otimes a_V \otimes S)(\ide_{H\otimes H} \otimes \Psi_{H,V})\\
&\quad \circ(m\otimes \Delta\otimes \ide_V)(\ide_H \otimes \Psi_{H,H}\otimes \ide_V)(\Delta\otimes \ide_{H\otimes V}),\\
\delta^{\rR}&=\Delta\otimes \ide_V,
\end{align*}
both pictured in \eqref{eq:aRdR} below. For $f\colon V\to W$ in $\lmod{H}(\cB)$, we have that $\rR_\cB(f)=\ide_H\otimes f$. 

The lax monoidal structure is given by the morphism $u \colon I \overset{\alpha_I}{\longrightarrow} \rR_{\cB}\rF_\cB(I)\overset{\sim}{\rightarrow} \rR_{\cB}(I)$ and the natural transformation $\tau_{V,W}\colon \rR_\cB(V)\otimes \rR_\cB(W)\to \rR_\cB(V\otimes W)$ defined by
\begin{align*}
\tau_{V,W}&=(m_H \otimes \ide_{V\otimes W})(\ide_{H} \otimes \Psi_{V,H}\otimes \ide_W).
\end{align*}
\end{theorem}

\begin{proof}
Checking that $\rR_\cB(V)$ is an object in the category of $H$-Yetter--Drinfeld modules in $\cB$ is carried out using graphical calculus, especially since Sweedler notation cannot be employed easily for objects of $\cB$;  this argument is quite similar to \cite{Dav2}*{Proposition 5.1} for the case when $\cB$ is symmetric monoidal. Here,  the $H$-action and $H$-coaction on $\rR_{\cB}(V)$ is displayed as follows:  
\smallskip
 
\begin{align}  \label{eq:aRdR}
a^{\rR}&=\vcenter{\hbox{
\begingroup%
  \makeatletter%
  \providecommand\color[2][]{%
    \errmessage{(Inkscape) Color is used for the text in Inkscape, but the package 'color.sty' is not loaded}%
    \renewcommand\color[2][]{}%
  }%
  \providecommand\transparent[1]{%
    \errmessage{(Inkscape) Transparency is used (non-zero) for the text in Inkscape, but the package 'transparent.sty' is not loaded}%
    \renewcommand\transparent[1]{}%
  }%
  \providecommand\rotatebox[2]{#2}%
  \newcommand*\fsize{\dimexpr\f@size pt\relax}%
  \newcommand*\lineheight[1]{\fontsize{\fsize}{#1\fsize}\selectfont}%
  \ifx\svgwidth\undefined%
    \setlength{\unitlength}{37.84330664bp}%
    \ifx\svgscale\undefined%
      \relax%
    \else%
      \setlength{\unitlength}{\unitlength * \real{\svgscale}}%
    \fi%
  \else%
    \setlength{\unitlength}{\svgwidth}%
  \fi%
  \global\let\svgwidth\undefined%
  \global\let\svgscale\undefined%
  \makeatother%
  \begin{picture}(1,1.97754772)%
    \lineheight{1}%
    \setlength\tabcolsep{0pt}%
    \put(0,0){\includegraphics[width=\unitlength,page=1]{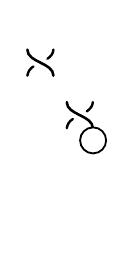}}%
    \put(0.59811281,0.8362876){\color[rgb]{0,0,0}\makebox(0,0)[lt]{\lineheight{0}\smash{\begin{tabular}[t]{l}$+$\end{tabular}}}}%
    \put(0,0){\includegraphics[width=\unitlength,page=2]{RBaction.pdf}}%
    \put(0.01016711,1.89724124){\color[rgb]{0,0,0}\makebox(0,0)[lt]{\lineheight{1.25}\smash{\begin{tabular}[t]{l}$H$\end{tabular}}}}%
    \put(0.30744557,1.89724124){\color[rgb]{0,0,0}\makebox(0,0)[lt]{\lineheight{1.25}\smash{\begin{tabular}[t]{l}$H$\end{tabular}}}}%
    \put(0.60472404,1.89724124){\color[rgb]{0,0,0}\makebox(0,0)[lt]{\lineheight{1.25}\smash{\begin{tabular}[t]{l}$V$\end{tabular}}}}%
    \put(0.60472404,0.01553486){\color[rgb]{0,0,0}\makebox(0,0)[lt]{\lineheight{1.25}\smash{\begin{tabular}[t]{l}$V$\end{tabular}}}}%
    \put(0.20835275,0.01553486){\color[rgb]{0,0,0}\makebox(0,0)[lt]{\lineheight{1.25}\smash{\begin{tabular}[t]{l}$H$\end{tabular}}}}%
  \end{picture}%
\endgroup%
}}, & \delta^{\rR}&=\vcenter{\hbox{
\begingroup%
  \makeatletter%
  \providecommand\color[2][]{%
    \errmessage{(Inkscape) Color is used for the text in Inkscape, but the package 'color.sty' is not loaded}%
    \renewcommand\color[2][]{}%
  }%
  \providecommand\transparent[1]{%
    \errmessage{(Inkscape) Transparency is used (non-zero) for the text in Inkscape, but the package 'transparent.sty' is not loaded}%
    \renewcommand\transparent[1]{}%
  }%
  \providecommand\rotatebox[2]{#2}%
  \newcommand*\fsize{\dimexpr\f@size pt\relax}%
  \newcommand*\lineheight[1]{\fontsize{\fsize}{#1\fsize}\selectfont}%
  \ifx\svgwidth\undefined%
    \setlength{\unitlength}{37.16210967bp}%
    \ifx\svgscale\undefined%
      \relax%
    \else%
      \setlength{\unitlength}{\unitlength * \real{\svgscale}}%
    \fi%
  \else%
    \setlength{\unitlength}{\svgwidth}%
  \fi%
  \global\let\svgwidth\undefined%
  \global\let\svgscale\undefined%
  \makeatother%
  \begin{picture}(1,0.80396279)%
    \lineheight{1}%
    \setlength\tabcolsep{0pt}%
    \put(0,0){\includegraphics[width=\unitlength,page=1]{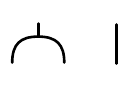}}%
    \put(0.19288379,0.72218426){\color[rgb]{0,0,0}\makebox(0,0)[lt]{\lineheight{1.25}\smash{\begin{tabular}[t]{l}$H$\end{tabular}}}}%
    \put(0.7983392,0.72218426){\color[rgb]{0,0,0}\makebox(0,0)[lt]{\lineheight{1.25}\smash{\begin{tabular}[t]{l}$V$\end{tabular}}}}%
    \put(-0.00893467,0.01581963){\color[rgb]{0,0,0}\makebox(0,0)[lt]{\lineheight{1.25}\smash{\begin{tabular}[t]{l}$H$\end{tabular}}}}%
    \put(0.39470226,0.01581963){\color[rgb]{0,0,0}\makebox(0,0)[lt]{\lineheight{1.25}\smash{\begin{tabular}[t]{l}$H$\end{tabular}}}}%
    \put(0.7983392,0.01689622){\color[rgb]{0,0,0}\makebox(0,0)[lt]{\lineheight{1.25}\smash{\begin{tabular}[t]{l}$V$\end{tabular}}}}%
  \end{picture}%
\endgroup%
}}~.
\end{align}
\smallskip

We leave verification of the action condition $a^{\rR}(\ide_H \otimes a^{\rR}) = a^{\rR}(m \otimes \ide_{H \otimes V})$ to the reader. The Yetter--Drinfeld condition, Equation~\ref{eq:HYDB}, is verified by the following graphical calculation. 

\begin{align*}
\vcenter{\hbox{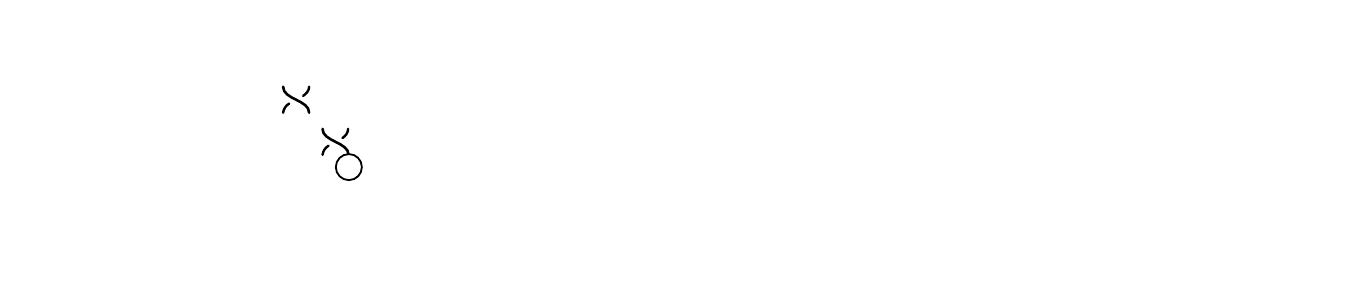}}
\end{align*}

\medskip

\noindent 
Here, the first and last equalities follow from  \eqref{eq:aRdR}. The second and third equality use coassociativity, the bialgebra axiom, and naturality of the braiding. The fourth equality follows from a computation using the antipode axioms and bialgebra condition while the fifth equality again uses naturality of the braiding and coassociativity. 

Functoriality of $\rR_{\cB}$ is clear by definition.
To see that $\rR_{\cB}$ is right adjoint to $\rF_\cB$, we present the unit $\alpha$ and counit $\beta$ of the adjunction. For objects $V$ of $\lmod{H}(\cB)$ and $W$ of $\lYD{H}(\cB)$, define
\begin{equation} \label{eq:alphabeta}
\alpha_W:=\delta_W \colon W\rightarrow \rR_\cB\rF_\cB(W) \quad \quad \text{and} \quad \quad
\beta_V:= \varepsilon\otimes \ide_V\colon \rF_\cB\rR_{\cB}(V)\rightarrow V.
\end{equation}
A direct check verifies the adjunction axioms for $\alpha,\beta$, and further one can check directly that $\alpha_V$ is a morphism in  $\lYD{H}(\cB)$ and $\beta_V$ a morphism in $\lmod{H}(\cB)$.

The lax monoidal structure is computed as in \cite{Dav1}*{Section 5} as
$$\tau_{V,W}=\rR_\cB(\beta_V\otimes \beta_W)\; \rR_\cB((\rF_\cB)_{V,W}) \;(\alpha_{\rR_\cB(V)\otimes \rR_\cB(W)}).$$
Using \eqref{eq:alphabeta} and omitting associativity, we have that $\tau_{V,W}\colon \rR_\cB(V) \otimes \rR_\cB(W) \to \rR_\cB(V \otimes W)$ is
\[
\begin{array}{l}
\tau_{V,W} = (\ide_H \otimes \;\varepsilon \otimes \ide_V \otimes \;\varepsilon \otimes \ide_W)\delta_{H \otimes V \otimes H \otimes W} \\
 ~~= (\ide_H \otimes\; \varepsilon \otimes \ide_V \otimes\; \varepsilon \otimes \ide_W)(m \otimes \ide_{H\otimes V \otimes H\otimes W})(\ide_H \otimes \Psi_{H \otimes V, H} \otimes \ide_{H \otimes W})(\Delta \otimes \ide_V \otimes \Delta \otimes \ide_W)\\
~~=(m \otimes \ide_{V \otimes W})(\ide_H \otimes \Psi_{V,H} \otimes \ide_W) .
\end{array}
\]
We have to verify associativity and unitarity squares for this lax monoidal structure, and these follow directly from the corresponding properties of $H$. It is also directly verified that $u$ and $\tau_{V,W}$ are indeed morphisms in $\lYD{H}(\cB)$. \end{proof}

For any lax monoidal functor $G\colon \cC\to \cD$ with $A$ an algebra object in $\cC$, we get that $G(A)$ is an algebra in $\cD$ with product $m_{G(A)}=G(m_A) G_{A,A}.$ Using this, we observe the following:

\begin{corollary} \label{cor:RBA}
The lax monoidal  functor $\rR_\cB$ induces the following functor 
\begin{align*}
\rR_\cB\colon \Alg(\lmod{H}(\cB))\rightarrow \Alg(\lYD{H}(\cB))
\end{align*}
on categories of algebra objects. Given an algebra $A$ with product $m_A$ in $\lmod{H}(\cB)$, the algebra $\rR_\cB(A)$ has product given by
\begin{align*}
m=(m_H\otimes m_A)(\ide_H\otimes \Psi_{A,H}\otimes\ide_A)=(\ide_H\otimes m_A)\tau_{A,A},
\end{align*}
that is, given by the tensor product algebra structure on $H\otimes A$ in $\cB$. \qed
\end{corollary}

\begin{example}\label{adjoint-expl}
Consider  the $\otimes$-unit $I$ of $\cB$. Then $I$ is an $H$-module in $\cB$ with trivial action, $a_I =\varepsilon\otimes \ide_I$, and moreover, $I \in \Alg(\lmod{H}(\cB))$. The Yetter--Drinfeld module $\rR_\cB(I)=H\otimes I\cong H$ is given by the adjoint $H$-action
\begin{align*}
a_{\op{ad}}
= m(m\otimes S)(\ide_H\otimes \Psi_{H,H})(\Delta\otimes \ide_H),
\end{align*} 
together with the regular coaction $\Delta$. It follows that $H$, with this Yetter--Drinfeld structure, is an algebra object in $\lYD{H}(\cB)$. 
\end{example}

\begin{example}\label{coreg-expl}
Suppose that we have a Hopf algebra $\rdual{H}$ in $\cB$ right dually paired to a Hopf algebra $H$ in $\cB$, that is, there is a pairing $\ev\colon H\otimes \rdual{H} \to I$ in $\cB$ satisfying the conditions of \cite{L17}*{Definition~3.1}. Extending the graphical calculus by denoting $\ev=\vcenter{\hbox{
\begingroup%
  \makeatletter%
  \providecommand\color[2][]{%
    \errmessage{(Inkscape) Color is used for the text in Inkscape, but the package 'color.sty' is not loaded}%
    \renewcommand\color[2][]{}%
  }%
  \providecommand\transparent[1]{%
    \errmessage{(Inkscape) Transparency is used (non-zero) for the text in Inkscape, but the package 'transparent.sty' is not loaded}%
    \renewcommand\transparent[1]{}%
  }%
  \providecommand\rotatebox[2]{#2}%
  \ifx\svgwidth\undefined%
    \setlength{\unitlength}{16.80101968bp}%
    \ifx\svgscale\undefined%
      \relax%
    \else%
      \setlength{\unitlength}{\unitlength * \real{\svgscale}}%
    \fi%
  \else%
    \setlength{\unitlength}{\svgwidth}%
  \fi%
  \global\let\svgwidth\undefined%
  \global\let\svgscale\undefined%
  \makeatother%
  \begin{picture}(1,0.52545969)%
    \put(0,0){\includegraphics[width=\unitlength]{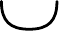}}%
  \end{picture}%
\endgroup%
}}$, 
we define the \emph{left coregular action} $a^{\mathrm{cor}}$ of $H$ on $\rdual{H}$ by
\begin{align*}
a^{\mathrm{cor}}=(\ev\otimes \ide_{\rdual{H}})(\ide_H\otimes \Delta_{\rdual{H}})=\vcenter{\hbox{
\begingroup%
  \makeatletter%
  \providecommand\color[2][]{%
    \errmessage{(Inkscape) Color is used for the text in Inkscape, but the package 'color.sty' is not loaded}%
    \renewcommand\color[2][]{}%
  }%
  \providecommand\transparent[1]{%
    \errmessage{(Inkscape) Transparency is used (non-zero) for the text in Inkscape, but the package 'transparent.sty' is not loaded}%
    \renewcommand\transparent[1]{}%
  }%
  \providecommand\rotatebox[2]{#2}%
  \newcommand*\fsize{\dimexpr\f@size pt\relax}%
  \newcommand*\lineheight[1]{\fontsize{\fsize}{#1\fsize}\selectfont}%
  \ifx\svgwidth\undefined%
    \setlength{\unitlength}{72.24902333bp}%
    \ifx\svgscale\undefined%
      \relax%
    \else%
      \setlength{\unitlength}{\unitlength * \real{\svgscale}}%
    \fi%
  \else%
    \setlength{\unitlength}{\svgwidth}%
  \fi%
  \global\let\svgwidth\undefined%
  \global\let\svgscale\undefined%
  \makeatother%
  \begin{picture}(1,0.54296257)%
    \lineheight{1}%
    \setlength\tabcolsep{0pt}%
    \put(0,0){\includegraphics[width=\unitlength,page=1]{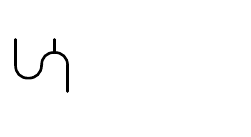}}%
    \put(-0.01148913,0.43780328){\color[rgb]{0,0,0}\makebox(0,0)[lt]{\lineheight{1.25}\smash{\begin{tabular}[t]{l}$H$\end{tabular}}}}%
    \put(0.14422232,0.43780328){\color[rgb]{0,0,0}\makebox(0,0)[lt]{\lineheight{1.25}\smash{\begin{tabular}[t]{l}$\rdual{H}$\end{tabular}}}}%
    \put(0.15460309,0.02257275){\color[rgb]{0,0,0}\makebox(0,0)[lt]{\lineheight{1.25}\smash{\begin{tabular}[t]{l}$\rdual{H}$\end{tabular}}}}%
    \put(0.37259911,0.1782842){\color[rgb]{0,0,0}\makebox(0,0)[lt]{\lineheight{1.25}\smash{\begin{tabular}[t]{l}$.$\end{tabular}}}}%
  \end{picture}%
\endgroup%
}}
\end{align*}
The following computation shows that $\rdual{H}^{\Psi^{-1}}$ is an algebra object in $\lmod{H}(\cB)$.

\begin{align*}
\vcenter{\hbox{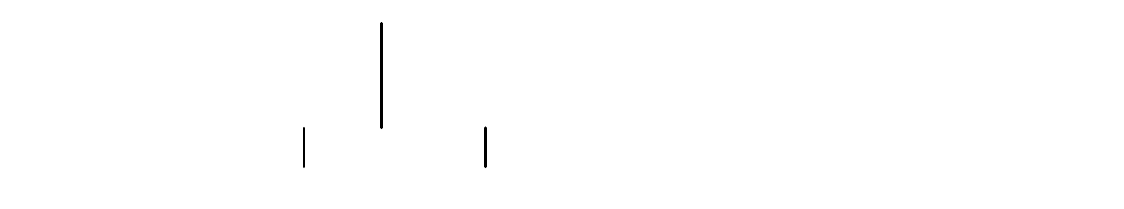}}
\end{align*}

\noindent Hence, $\rR_\cB(\rdual{H}^{\Psi^{-1}})=H\otimes\rdual{H}^{\Psi^{-1}}$ 
is an algebra object in $\lYD{H}(\cB)$.

\end{example}

\subsection{The \texorpdfstring{$\cB$}{B}-center} \label{sec:Bcenter}

This section contains the categorical definition of the $\cB$-center, which is a direct generalization of the \emph{full center} of  Davydov's works \cites{Dav1,Dav2} relative to a braided monoidal category $\cB$. Davydov's case corresponds to specializing $\cB=\Vect$.

\begin{definition} \label{def:Bcenter}
Let $A$ be an algebra in a $\cB$-central monoidal category $\cC$. The \emph{$\cB$-center} of $A$ is a pair ($Z_\cB(A)$, $\zeta_A$), where $Z_\cB(A)$ is an object in $\cZ_\cB(\cC)$ with half-braiding $c_{Z_\cB(A),A}$ and $\zeta_A := \zeta^\cB_A \colon Z_\cB(A)\to A$ is a morphism in $\cC$, which is terminal among pairs $((Z, c_{Z,A}),\zeta\colon Z\to A)$ in the comma category $\rF_\cB \hspace{-.05in} \downarrow \hspace{-.05in} A$ so that the following diagram commutes.
\begin{align}\label{Bcenter-diag}
\xymatrix@R=1.5pc@C=3pc{
Z\otimes A\ar[rr]^-{c_{Z,A}}\ar[d]^{\zeta\otimes \ide}&&A\otimes Z\ar[d]^{\ide\otimes \zeta}\\
A\otimes A\ar[rd]_m&& A\otimes A\ar[ld]^m\\
&A&
}
\end{align}
When $\cB = \Vect$, the pair $(Z(A):=Z_{\Vect}(A),~\zeta_A^{\Vect})$ is  the {\it full center} of $A$ (as in \cites{Dav1,Dav2}).
\end{definition}

The $\cB$-center of $A$ is realized as a terminal object of the following braided monoidal category.

\begin{definition} \label{def:ZBA}
Let $A$ be an algebra in $\cC$. We denote by $\cal{Z}_\cB(A)$ the category consisting of
 \begin{itemize}
 \item pairs $(Z,\zeta)$ with $Z=(Z,c)$ an object in $\cZ_\cB(\cC)$, and $\zeta\colon Z\to A$ a morphism in $\cC$, that make Diagram~\ref{Bcenter-diag} commute; and
\item morphisms $(Z,\zeta)\to (Z',\zeta')$ in $\cZ_\cB(\cC)$  that are morphisms $f\colon Z \to Z'$ such that the diagram below commutes.
\begin{align*}
\xymatrix{Z\ar^f[rr] \ar_{\zeta}[dr]&&Z'\ar^{\zeta'}[dl]\\
&A&
}
\end{align*}
\end{itemize}
\end{definition}

Given objects $(Z,\zeta)$, $(Z',\zeta') \in \cZ_\cB(A)$, their tensor product is $(Z\otimes Z',m(\zeta\otimes \zeta'))$, using the tensor product $Z\otimes Z'$ in $\cZ_\cB(\cC)$, cf. \cite{Dav1}*{Remark~4.2}. This makes $\cZ_\cB(A)$ a monoidal category.

\smallskip

The category $\cZ_\cB(A)$ is braided via 
$$\Psi_{(Z,\zeta),(Z',\zeta')}:=c_{Z,Z'}\colon Z\otimes Z'\rightarrow Z'\otimes Z,$$
which is a morphism in $\cZ_\cB(A)$ by commutativity of the outer diagram in

\begin{align*}
\xymatrix{
Z\otimes Z'\ar[rrrr]^{c_{Z,Z'}}\ar[rd]^{\ide\otimes \zeta'}\ar[dd]_{\zeta\otimes \zeta'} &&&& Z'\otimes Z\ar[ld]_{\zeta'\otimes \ide}\ar[dd]^{\zeta'\otimes \zeta}\\
&Z\otimes A\ar[ld]^{\zeta\otimes \ide}\ar[rr]^{c_{Z,A}}&&A\otimes Z\ar[rd]_{\ide\otimes \zeta}&\\
A\otimes A\ar[rr]_m &&A&& \ar[ll]^m A\otimes A.}
\end{align*}

\smallskip

\noindent The upper middle diagram commutes by naturality of $c$. We have the following result.

\begin{proposition} \label{prop:ZBA}
The $\cB$-center $Z_\cB(A)$ of $A\in \Alg(\cC)$ is a braided commutative algebra  in $\cZ_\cB(\cC)$ and  $\zeta_A\colon Z_\cB(A) \to A$ is a morphism of algebras in $\cC$. 
 \end{proposition}
\begin{proof}
It follows from Lemma \ref{lem:terminal} that $(Z_\cB(A),\zeta_A)$, being the terminal object of the braided monoidal category $\cZ_\cB(A)$, is a commutative algebra in $\cZ_\cB(A)$. Since the forgetful functor $\cZ_\cB(A)\to \cZ_\cB(\cC)$ is a braided monoidal functor, we get that $Z_\cB(A)$ is a commutative algebra in $\cZ_\cB(\cC)$. 
Moreover, the product $m_{Z_\cB(A)}$ is a morphism in $\cZ_\cB(A)$. So,
$$\zeta_A m_{Z_\cB(A)}=m_A(\zeta_A\otimes \zeta_A),$$
and this condition means that $\zeta_A$ is a morphism of algebras in $\cC$.
\end{proof}

\begin{corollary}\label{morphismtoZ}
For any algebra $A$ in $\cC$, there is a unique morphism of algebras $\xi_A\colon Z_{\cB}(A)\to Z(A)$ in $\cZ(\cC)$, which commutes with the respective morphisms of algebras to $A$.
\end{corollary}

\begin{proof}
Recall from Proposition \ref{prop:rel-center}(3) that $\cZ_{\cB}(\cC)$ is a braided monoidal subcategory of $\cZ(\cC)$. This implies that $\cZ_\cB(A)$ is a braided monoidal subcategory of $\cZ(A)$. Hence, $Z_\cB(A)$ is an algebra in $\cZ(A)\:=\cZ_{\Vect}(A)$.
By Lemma \ref{lem:terminal2}, we see that the unique morphism $\xi_A\colon Z_{\cB}(A)\to Z(A)$ is one of algebras in $\cZ(A)$. In particular, $\xi_A$ is a morphism of algebras in $\cZ(\cC)$, which commutes with the respective morphisms to $A$ in the sense that the following diagram commutes.
\[
\xymatrix{
Z_{\cB}(A)\ar[rr]^{\xi_A}\ar[dr]_{\zeta_A^{\cB}}&& Z(A)\ar[dl]^{\zeta_A^{\Vect}}\\
&A&
}
\]

\vspace{-.25in}
\end{proof}

\subsection{The \texorpdfstring{$\cB$}{B}-center as a left center} \label{sec:Bctr-lfctr}
We will now show that the $\cB$-center of an algebra in a $\cB$-central monoidal category can be computed as the left center of its image under the functor $\rR_{\cB}$, thus generalizing \cite{Dav1}*{Theorem~5.4}.

\begin{theorem} \label{thm:Bcenter}
For a $\cB$-central monoidal category $\cC$, assume that there exists a right adjoint $\rR_\cB$ to the forgetful functor $\cZ_\cB(\cC)\to \cC$, and that the counit is given by epimorphisms. Let $A\in \Alg(\cC)$. Then, there is a canonical isomorphism of (commutative) algebras $C^l(\rR_\cB(A))\cong Z_\cB(A)$ in $\cZ_\cB(\cC)$.
\end{theorem}
\begin{proof}
Given the $\cB$-central set-up provided in previous sections, the proof of the theorem for the relative monoidal center $\cZ_\cB(\cC)$ is now analogous to Davydov's formal proof for $\cZ(\cC)$ in \cite{Dav1}*{Theorem~5.4}. The proof crucially uses the hypothesis that $\beta_A$ is an epimorphism. 
\end{proof}

\subsection{Morita invariants} \label{sec:Morita} 
Next, we turn our attention to module categories over the monoidal categories discussed above. A {\it left module category} over a monoidal category $\cC$ is a category  $\cM$ equipped with an bifunctor $\ast\colon \cC \times \cM \to \cM$  and natural isomorphisms for associativity $$\{m_{X,Y,M}\colon (X\otimes Y) \ast M \overset{\sim}{\to} X \ast (Y \ast M) ~|~ X,Y \in \cC,~ M \in \cM\}$$ and for unitality, which are compatible with the structure of $\cM$; see \cite{EGNO}*{Section~7.1} for details. A morphism between two $\cC$-module categories $\cM$ and $\cN$ is a functor $F\colon \cM \to \cN$ equipped with natural isomorphisms
$$\{s_{X,M} = s_{X,M}^{F}\colon F(X \ast M) \overset{\sim}{\to} X \ast F(M) ~|~ X \in \cC, ~M \in \cM\}$$
that are compatible with the associativity and unitality structure of $\cM$, cf. \cite{EGNO}*{Section~7.2}. The collection of $\cC$-module endofunctors of a $\cC$-module category $\cM$ is a monoidal category and is denoted by $\cEnd_\cC(\cM)$.

\smallskip

For an algebra $A \in \cC$, recall that $\rmod{A}(\cC)$  is the category of right modules over $A$ in $\cC$. It is a left $\cC$-module category via $X \ast (M, \rho) := (X \otimes M,~ \text{Id}_X \otimes \rho)$ for all $X \in \cC$ and $M \in \rmod{A}(\cC)$ with structure morphism $\rho\colon M \otimes A \to M$ in $\cC$.
\smallskip

\begin{definition}
We say that two algebras $A$ and $A'$ in a monoidal category $\cC$ are {\it Morita equivalent} if $\rmod{A}(\cC)$ and $\rmod{A'}(\cC)$ are equivalent as left $\cC$-module categories.
\end{definition}

The above generalizes the notion of Morita equivalence for rings or for algebras in $ \Vect$. We will establish the following result later in this section. 

\begin{theorem} \label{thm:Morita}
Take $\cC$ a $\cB$-central monoidal category, and let $A$ and $A'$ be algebras in~$\cC$. Suppose that $\rmod{A}(\cC)$ and $\rmod{A'}(\cC)$ are equivalent as left $\cC$-module categories. Then, the $\cB$-centers $Z_\cB(A)$ and $Z_\cB(A')$ are isomorphic as (commutative) algebras in $\cZ_{\cB}(\cC)$. In particular, the $\cB$-center of an algebra in $\cC$ is a Morita invariant.
\end{theorem}

This result is a generalization of \cite{Dav1}*{Theorem~6.2 and Corollary~6.3} in the case when $\cB = \Vect$, and see the discussion in  Remark~\ref{rem:cent} in the next section for an example of how it can be used in practice. For the proof of the theorem above, we need the next construction.

\begin{definition} \label{def:ZBM} Take $\cC$ a $\cB$-central monoidal category, and let $\cM$ be a left $\cC$-module category. Consider the monoidal functor $$E\colon \cZ_\cB(\cC) \to \cEnd_\cC(\cM), \qquad (V, c_{V,-}) \mapsto (L_V,s^{L_V}),$$ where
$L_V\colon \cM \to \cM$ is the functor given by $M \mapsto V \ast M$, and $s$ is the collection of natural isomorphisms,  for each $X \in \cC$ and $M \in \cM$, given by
$$s_{X,M}^{L_V}:=m_{X,V,M} \; (c_{V,X} \ast \text{Id}_M) \;m_{V,X,M}^{-1}\colon L_V(X \ast M) \overset{\sim}{\to} X \ast L_V(M).$$ Then the {\it $\cB$-center of $\cM$} is defined to be the terminal object in the comma category $E \hspace{-.04in} \downarrow \hspace{-.04in} I_{\cEnd_\cC(\cM)}$.

\smallskip

Namely, $Z_{\cB}(\cM)$ is the terminal object 
amongst pairs $((Z,c_{Z,-}), z)$ for $(Z,c_{Z,-}) \in \cZ_{\cB}(\cC)$ and $\Set{z = z_M\colon Z \ast M \to M}_{\cM\in \cM}$ a natural transformation, such that for all $X \in \cC$, $M\in \cM$, the following diagram commutes. 
\begin{align*}
\xymatrix@R=1.5pc@C=3pc{(Z \otimes X) \ast M \ar[rr]^-{c_{Z,X} \ast \text{Id}_M}\ar[d]_{m_{Z,X,M}}&&
(X \otimes Z) \ast M 
\ar[d]^{m_{X,Z,M}}\\
Z \ast (X \ast M) \ar[rd]_{z_{X \ast M}}&& X \ast (Z \ast M) \ar[ld]^{~\text{Id}_X \ast z_M}\\
&X \ast M&
}
\end{align*}
\end{definition}

\smallskip

\begin{proof}[Proof of Theorem~\ref{thm:Morita}]
 Analogous to the proof of Proposition~\ref{prop:ZBA}, we first have that $Z_\cB(\cM)$ is a commutative algebra in $\cZ_\cB(\cC)$. Moreover, if $\cM$ and $\cM'$ are equivalent $\cC$-module categories, then $Z_\cB(\cM) \cong Z_\cB(\cM')$ in $\ComAlg(\cZ_\cB(\cC))$. So, it suffices to establish  that $Z_\cB(\rmod{A}(\cC)) \cong Z_\cB(A)$ in $\ComAlg(\cZ_\cB(\cC))$. In turn, it suffices to show that the comma category $E \hspace{-.04in} \downarrow \hspace{-.04in} I_{\cEnd_\cC(\rmod{A}(\cC))}$ used in Definition~\ref{def:ZBM} is monoidally equivalent to the  category $\cZ_\cB(A)$ from Definition~\ref{def:ZBA}. At this point, one can proceed exactly as in the proof of \cite{Dav1}*{Theorem~6.2} using only the half-braidings of the full braided monoidal subcategory $\cZ_\cB(\cC)$ of $\cZ(\cC)$ in order to finish the proof.
\end{proof}

The converse of Theorem~\ref{thm:Morita} holds when $\cB=\Vect$, with $\cC$ a (braided monoidal) modular tensor category,  and the algebras in $\cC$ in question being simple and non-degenerate, by \cite{KR}*{Section~4.4}. So we ask:

\begin{question} \label{ques:Morita-conv}
In general, what conditions do we need on $\cC$, on $\cB$, and on algebras in $\cC$ for a converse of Theorem~\ref{thm:Morita} to hold?
\end{question}

We discuss the special setting of when $\cC$ is braided  next.

\subsection{The case when \texorpdfstring{$\cC$}{C} is braided} \label{sec:Cbraided}
As mentioned in Example~\ref{ex:augment}(2) and Proposition~\ref{prop:rel-center}(1), we have that $\cB$ is $\cB$-central, and that $\cZ_\cB(\cB)$ is isomorphic to $\cB$ as braided monoidal categories. 
For instance, take $\cC = \lmod{\Bbbk}(\cB)$ --- this can be identified canonically with $\cB$ and is isomorphic to $\cZ_\cB(\cB)$ as braided monoidal categories. Moreover in this case, there is a natural isomorphism $\rR_\cB \overset{\sim}{\implies} \ide_\cB$, and the $\cB$-center of $A \in \Alg(\cB)$ is given by $Z_\cB(A) \cong C^l(\rR_\cB(A)) \cong C^l(A)$ as commutative algebras in $\cZ_\cB(\cB)$ by Theorem~\ref{thm:Bcenter}.

\smallskip

So, when $\cC$ is braided, one does not need to work outside of $\cC$ to get Morita invariants for algebras in $\cC$ via Theorem~\ref{thm:Morita}, as $\cC$ is isomorphic to its relative monoidal center. This is computationally more feasible than working with the full center construction of \cite{Dav1}; see, for example, Remark~\ref{rem:gens}.
In particular, constructing Morita invariants of algebras in modular tensor categories  was one of the motivations behind Davydov's work \cite{Dav1} and other previous works \cites{FFRS, KR} --- now our construction of $\cB$-centers makes this goal more tractable computationally. 


\medskip

\section{Connection to centralizer algebras} \label{sec:cent}

In this section, we restrict our attention to the situation where $H$ is a Hopf algebra in a braided monoidal category $\cB$, and $\cC=\lmod{H}(\cB)$. Take $A \in \Alg(\cC)$. We saw in Theorem \ref{thm:Bcenter} that the $\cB$-center $Z_\cB(A)$ can be computed as the left center of $\rR_\cB(A)$, and we will now realize $Z_\cB(A)$ as a braided version of a centralizer algebra --- see Theorem~\ref{cent-thm} below. We begin by discussing braided smash product algebras in Section~\ref{sec:smash}. Then Theorem~\ref{cent-thm} is established in Section~\ref{sec:cent-result}, and consequences of this result are provided in Section~\ref{sec:cent-conseq}.

\subsection{Braided smash product algebras} \label{sec:smash} Consider the following terminology. 

\begin{definition} \label{def:smash} \cite{Majid-cross}*{Proposition~2.3}
Take $H$ a Hopf algebra in $\cB$, and $A$ an algebra in $\lmod{H}(\cB)$. The {\it (braided) smash product algebra} or {\it (braided) crossed product algebra} $A \rtimes H$ is the algebra in $\cB$ that is $A \otimes H$  as an object in $\cB$ equipped with multiplication $$m_{A \rtimes H} = (m_A \otimes m_H)(\ide_A \otimes a_A \otimes \ide_{H \otimes H})(\ide_{A \otimes H} \otimes \Psi_{H,A} \otimes \ide_H)(\ide_A \otimes \Delta_H \otimes \ide_{A \otimes H})$$ and with unit $u_{A \rtimes H} = u_A \otimes u_H$.
\end{definition}

Moreover, the result below describes the category of modules over braided smash product algebras.

\begin{proposition} \label{prop:smash} \cite{Majid-cross}*{Proposition~2.7}
There is an equivalence of monoidal categories  $$\lmod{A}(\lmod{H}(\cB)) \simeq \lmod{A\rtimes H}(\cB),$$ where an object $V$ of $\lmod{A}(\lmod{H}(\cB))$ with left $H$-action $a_V^H$ and left $A$-action  $a_V^A$ gets sent to the object $V$ with $A \rtimes H$-action $a_{A \rtimes H} = a_V^A(\ide_A \otimes \; a_V^H)$. \qed
\end{proposition}

Next, we provide a preliminary result on braided smash product algebras.

\begin{definition}\label{phi-def}
We define a map $\varphi\colon H\otimes A\to A\rtimes H$ by 
\begin{align*}
\varphi~=~\Psi^{-1}_{H,A}(S^{-1}\otimes \ide_A)~=~(\ide_A \otimes S^{-1})\Psi^{-1}_{H,A}~=~\vcenter{\hbox{
\begingroup%
  \makeatletter%
  \providecommand\color[2][]{%
    \errmessage{(Inkscape) Color is used for the text in Inkscape, but the package 'color.sty' is not loaded}%
    \renewcommand\color[2][]{}%
  }%
  \providecommand\transparent[1]{%
    \errmessage{(Inkscape) Transparency is used (non-zero) for the text in Inkscape, but the package 'transparent.sty' is not loaded}%
    \renewcommand\transparent[1]{}%
  }%
  \providecommand\rotatebox[2]{#2}%
  \newcommand*\fsize{\dimexpr\f@size pt\relax}%
  \newcommand*\lineheight[1]{\fontsize{\fsize}{#1\fsize}\selectfont}%
  \ifx\svgwidth\undefined%
    \setlength{\unitlength}{25.87051449bp}%
    \ifx\svgscale\undefined%
      \relax%
    \else%
      \setlength{\unitlength}{\unitlength * \real{\svgscale}}%
    \fi%
  \else%
    \setlength{\unitlength}{\svgwidth}%
  \fi%
  \global\let\svgwidth\undefined%
  \global\let\svgscale\undefined%
  \makeatother%
  \begin{picture}(1,1.58972305)%
    \lineheight{1}%
    \setlength\tabcolsep{0pt}%
    \put(0,0){\includegraphics[width=\unitlength,page=1]{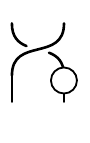}}%
    \put(0.5502787,0.5870745){\color[rgb]{0,0,0}\makebox(0,0)[lt]{\lineheight{0}\smash{\begin{tabular}[t]{l}$-$\end{tabular}}}}%
    \put(-0.01283435,1.47225099){\color[rgb]{0,0,0}\makebox(0,0)[lt]{\lineheight{1.25}\smash{\begin{tabular}[t]{l}$H$\end{tabular}}}}%
    \put(0.56697631,1.47225099){\color[rgb]{0,0,0}\makebox(0,0)[lt]{\lineheight{1.25}\smash{\begin{tabular}[t]{l}$A$\end{tabular}}}}%
    \put(-0.01283435,0.02272435){\color[rgb]{0,0,0}\makebox(0,0)[lt]{\lineheight{1.25}\smash{\begin{tabular}[t]{l}$A$\end{tabular}}}}%
    \put(0.56697631,0.02272435){\color[rgb]{0,0,0}\makebox(0,0)[lt]{\lineheight{1.25}\smash{\begin{tabular}[t]{l}$H$\end{tabular}}}}%
  \end{picture}%
\endgroup%
}}~,
\end{align*}
and note that $\varphi$ is an isomorphism in $\cB$, with $\varphi^{-1} = \Psi_{A,H}(\ide_A \otimes S) = (S \otimes \ide_A)\Psi_{A,H}$.
\end{definition}

\begin{lemma}\label{AH-YD} The isomorphism $\varphi$ defines on $A\rtimes H$ the structure of an $H$-Yetter--Drinfeld module in $\cB$ from such a structure on $\rR_\cB(A)=H\otimes A$ via Theorem~\ref{prop:RB-YD}. The $H$-action and $H$-coaction on $A\rtimes H$ are given by
\begin{align*}
a^{\rtimes}&= (\ide_A \otimes m_H)\Psi^{-1}_{H,A\otimes H} (S^{-1}\otimes a_{A\otimes H})(\Delta_H \otimes \ide_{A\otimes H}),\\
\delta^{\rtimes}&= (S\otimes \ide_{A\otimes H})\Psi_{A\otimes H,H} (\ide_A \otimes\Delta_H),
\end{align*}
using $a_{A\otimes H}$ from Equation~\ref{tensorproductaction}. 
\end{lemma}

\begin{proof}
The  action $a^{\rtimes}$ and the coaction $\delta^{\rtimes}$  are defined from the action $a^\rR$ and coaction $\delta^\rR$ in Theorem~\ref{prop:RB-YD} by requiring that $\varphi$ becomes a morphism of Yetter--Drinfeld modules. That is, 
\begin{align*}
 a^{\rtimes}&=\varphi a^{\rR}(\ide_H\otimes \varphi^{-1}), & \delta^{\rtimes}&=(\ide_H\otimes\varphi)\delta^{\rR}\varphi^{-1}.
\end{align*}
Since $a^{\rR},\delta^{\rR}$ are Yetter--Drinfeld compatible, $a^{\rtimes},\delta^{\rtimes}$ are also Yetter--Drinfeld compatible.
\end{proof}

\subsection{Main result} \label{sec:cent-result}
The following theorem generalizes \cite{Dav2}*{Theorem 5.3}.

\begin{theorem}\label{cent-thm}
For any algebra $A$ in $\lmod{H}(\cB)$, its $\cB$-center $Z_\cB(A)$ is isomorphic to $\Cent^l_{A\rtimes H}(A)^{\Psi^{-1}}$ as an algebra in $\cB$. 
\end{theorem}

Before proving the theorem, we need the following lemma.

\begin{lemma}\label{cent-lemma} Let $A$ be an algebra in $\cC=\lmod{H}(\cB)$. 
The left center $C^l(\rR_\cB(A))$ is the terminal object in the category of morphisms $\gamma\colon C\to H\otimes A$ in $\cB$ such that
\begin{align}\label{cent-cond1}
(\ide_H\otimes m_A)(\gamma\otimes \ide_A)&=(\ide_H\otimes m_A)(\ide_H\otimes a_A\otimes \ide_A)(\Delta_H\otimes \Psi_{A, A})(\gamma\otimes \ide_A),
\end{align}
or, equivalently,
\begin{equation}\label{cent-cond2}\begin{split}
(\ide_H \otimes m_A)(\Psi_{A,H} \otimes \ide_A)(\ide_A \otimes \gamma)=&(\ide_H \otimes  m_A\Psi^{-1}_{A,A}) (\ide_H\otimes a_A\Psi^{-1}_{A,H}\otimes \ide_A)\\
&\circ  (\Psi_{A,H} \otimes S^{-1}\otimes \ide_A) (\ide_A \otimes \Delta_H\otimes \ide_A)(\ide_A \otimes \gamma).\end{split}
\end{equation}
\end{lemma}

We will only need \eqref{cent-cond2} for the proof of Theorem~\ref{cent-thm}, and using graphical calculus \eqref{cent-cond2} is

\begin{align}\label{commutation-cond}
\vcenter{\hbox{
\begingroup%
  \makeatletter%
  \providecommand\color[2][]{%
    \errmessage{(Inkscape) Color is used for the text in Inkscape, but the package 'color.sty' is not loaded}%
    \renewcommand\color[2][]{}%
  }%
  \providecommand\transparent[1]{%
    \errmessage{(Inkscape) Transparency is used (non-zero) for the text in Inkscape, but the package 'transparent.sty' is not loaded}%
    \renewcommand\transparent[1]{}%
  }%
  \providecommand\rotatebox[2]{#2}%
  \newcommand*\fsize{\dimexpr\f@size pt\relax}%
  \newcommand*\lineheight[1]{\fontsize{\fsize}{#1\fsize}\selectfont}%
  \ifx\svgwidth\undefined%
    \setlength{\unitlength}{88.78140067bp}%
    \ifx\svgscale\undefined%
      \relax%
    \else%
      \setlength{\unitlength}{\unitlength * \real{\svgscale}}%
    \fi%
  \else%
    \setlength{\unitlength}{\svgwidth}%
  \fi%
  \global\let\svgwidth\undefined%
  \global\let\svgscale\undefined%
  \makeatother%
  \begin{picture}(1,0.84551917)%
    \lineheight{1}%
    \setlength\tabcolsep{0pt}%
    \put(0,0){\includegraphics[width=\unitlength,page=1]{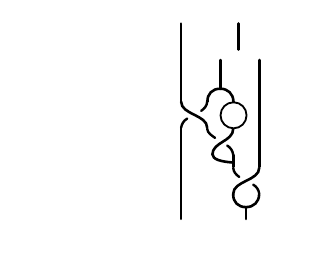}}%
    \put(0.71051687,0.44041743){\color[rgb]{0,0,0}\makebox(0,0)[lt]{\lineheight{0}\smash{\begin{tabular}[t]{l}$-$\end{tabular}}}}%
    \put(0.29283525,0.43114113){\color[rgb]{0,0,0}\makebox(0,0)[lt]{\lineheight{1.25}\smash{\begin{tabular}[t]{l}$=$\end{tabular}}}}%
    \put(0,0){\includegraphics[width=\unitlength,page=2]{leftcenter2.pdf}}%
    \put(0.73497768,0.61265667){\color[rgb]{0,0,0}\makebox(0,0)[lt]{\lineheight{0}\smash{\begin{tabular}[t]{l}$\gm$\end{tabular}}}}%
    \put(0.54536169,0.81128832){\color[rgb]{0,0,0}\makebox(0,0)[lt]{\lineheight{1.25}\smash{\begin{tabular}[t]{l}$A$\end{tabular}}}}%
    \put(0.72276371,0.81128832){\color[rgb]{0,0,0}\makebox(0,0)[lt]{\lineheight{1.25}\smash{\begin{tabular}[t]{l}$C$\end{tabular}}}}%
    \put(0.54536169,0.00875536){\color[rgb]{0,0,0}\makebox(0,0)[lt]{\lineheight{1.25}\smash{\begin{tabular}[t]{l}$H$\end{tabular}}}}%
    \put(0.75655457,0.00875536){\color[rgb]{0,0,0}\makebox(0,0)[lt]{\lineheight{1.25}\smash{\begin{tabular}[t]{l}$A$\end{tabular}}}}%
    \put(0,0){\includegraphics[width=\unitlength,page=3]{leftcenter2.pdf}}%
    \put(0.12558954,0.60838976){\color[rgb]{0,0,0}\makebox(0,0)[lt]{\lineheight{0}\smash{\begin{tabular}[t]{l}$\gm$\end{tabular}}}}%
    \put(-0.00373987,0.80915499){\color[rgb]{0,0,0}\makebox(0,0)[lt]{\lineheight{1.25}\smash{\begin{tabular}[t]{l}$A$\end{tabular}}}}%
    \put(0.12297592,0.80915499){\color[rgb]{0,0,0}\makebox(0,0)[lt]{\lineheight{1.25}\smash{\begin{tabular}[t]{l}$C$\end{tabular}}}}%
    \put(-0.00373987,0.00662178){\color[rgb]{0,0,0}\makebox(0,0)[lt]{\lineheight{1.25}\smash{\begin{tabular}[t]{l}$H$\end{tabular}}}}%
    \put(0.12297592,0.00662178){\color[rgb]{0,0,0}\makebox(0,0)[lt]{\lineheight{1.25}\smash{\begin{tabular}[t]{l}$A$\end{tabular}}}}%
  \end{picture}%
\endgroup%
}}\!\!. \end{align}

\begin{proof}[Proof of Lemma~\ref{cent-lemma}]
The left center $C^l(\rR_\cB(A))$ is the terminal object in the category of morphisms $\gamma\colon C\to H\otimes A$ in $\cB$ such that Diagram~\ref{Cleftdiag} commutes for $m= m_{\rR_\cB(A)}$. Since the multiplication on $\rR_\cB(A)$ is the tensor product multiplication on $H\otimes A$ [Corollary \ref{cor:RBA}], we have that 
Diagram~\ref{Cleftdiag} is equivalent to the first equality in 
\begin{equation*} \label{eq:cent-lem1}
m_{H\otimes A}(\gamma\otimes \ide_{H \otimes A}) ~=~ m_{H\otimes A}(\ide_{H \otimes A} \otimes \gamma)\Psi^{\mathrm{\bf YD}}_{C,H\otimes A} ~=~ 
m_{H\otimes A}\Psi^{\mathrm{\bf YD}}_{H\otimes A,H\otimes A}(\gamma\otimes \ide_{H \otimes A}).
\end{equation*}
The second equality uses the naturality of the Yetter--Drinfeld braiding applied to~$\gamma$. Pre-composition with the inverse Yetter--Drinfeld braiding gives the equivalent condition
\begin{align}\label{eq:cent-lem2}
m_{H\otimes A}(\ide_{H\otimes A}\otimes \gamma)&~=~ m_{H\otimes A}\left(\Psi^{\mathrm{\bf YD}}_{\rR_\cB(A),\rR_\cB(A)}\right)^{-1}(\ide_{H\otimes A}\otimes \gamma).
\end{align}

\noindent Recalling the Yetter--Drinfeld structure $a^{\rR}$, $\delta^{\rR}$ from Theorem \ref{prop:RB-YD}, we have by Equation \ref{eq:YDbraidinv} that
\begin{align*}
\left(\Psi^{\mathrm{\bf YD}}_{\rR_\cB(A),\rR_\cB(A)}\right)^{-1}=(\ide_{H\otimes A}\otimes (a^{\rR}\Psi_{H\otimes A,H}^{-1}))(\Psi^{-1}_{H\otimes A,H\otimes A}\otimes S^{-1})(\ide_{H\otimes A}\otimes (\Psi^{-1}_{H,H\otimes A}\delta^{\rR})).
\end{align*}

\noindent Therefore, using properties of the inverse antipode, we see that

\begin{align*}
m_{H\otimes A}\left(\Psi^{\mathrm{\bf YD}}_{\rR_\cB(A),\rR_\cB(A)}\right)^{-1}&=\vcenter{\hbox{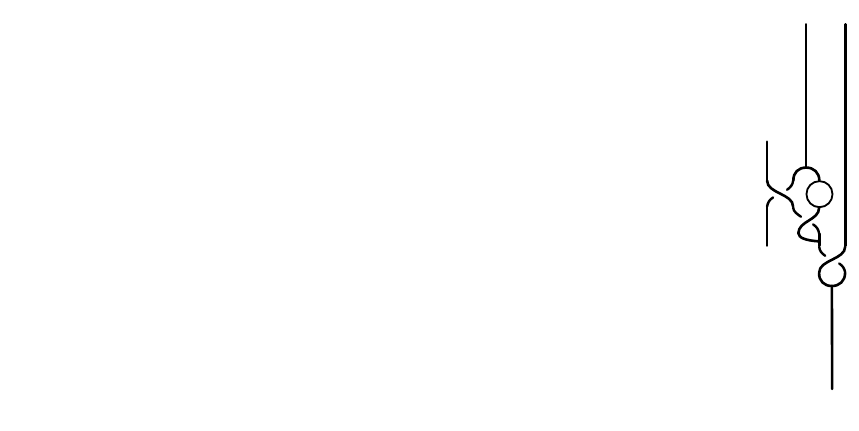}}.
\end{align*}
Thus, \eqref{eq:cent-lem2} is equivalent to 
\begin{align}\label{eq:cent-lem3}\begin{split}
m_{H\otimes A}(\ide_{H\otimes A}\otimes \gamma)&=(m_H\otimes m_A\Psi^{-1}_{A,A})(\ide_{H\otimes H}\otimes a_A\Psi_{A,H}^{-1}\otimes \ide_A)(\ide_H\otimes \Psi_{A,H}\otimes S^{-1}\otimes \ide_A)\\
&\quad \circ(\ide_{H\otimes A}\otimes \Delta\otimes \ide_A)(\ide_{H\otimes A}\otimes \gamma)
\end{split}
\end{align}
This condition is equivalent to \eqref{cent-cond2}:  
Indeed, pre-composing with $u_H\otimes \ide_A\otimes \ide_C$ yields  \eqref{cent-cond2}. Conversely, if \eqref{cent-cond2} holds, then  \eqref{eq:cent-lem3} follows  as the left copy of $H$ is multiplied last on the left in both sides of the equation.

Equation~\ref{cent-cond1} is shown to be equivalent to Equation \ref{eq:cent-lem2} using a similar argument.
\end{proof}

\medskip

\begin{proof}[Proof of Theorem \ref{cent-thm}]
The theorem will be proved via Theorem~\ref{thm:Bcenter} by showing that $C^l(\rR_\cB(A))$ is isomorphic to the  $\Psi^{-1}$-opposite algebra of the centralizer of $A$ inside of $A\rtimes H$. To start, recall Definition~\ref{def:scriptC}.

First, we show $\varphi$ from Definition \ref{phi-def} induces a $\Psi^{-1}$-opposite equivalence of monoidal categories 
$$\Gamma \colon \mathcal{C}^l(\rR_\cB(A))\rightarrow\mathcal{C}^l_{A\rtimes H}(A), \qquad  \Gamma(C,\gamma)=(C,\varphi\gamma), \qquad \Gamma(f)=f,$$ 
for all objects $(C, \gamma)$ and morphisms $f$ of $\mathcal{C}^l(\rR_\cB(A))$.
If $(C,\gamma)$ satisfies \eqref{commutation-cond}, then $\Gamma(C,\gamma)$ satisfies 
\begin{align}\label{eq4.9}
m_{A\rtimes H}(\ide_A\otimes \;u_H \otimes \varphi \gamma)&=m_{A\rtimes H}\Psi^{-1}_{A\rtimes H,A\rtimes H}(\ide_A\otimes \; u_H \otimes \varphi \gamma).
\end{align}
This follows from the following series of equalities of morphisms from $A\otimes C$ to $A\rtimes H$, which we display using graphical calculus:
\begin{align}\label{calc-center}
\vcenter{\hbox{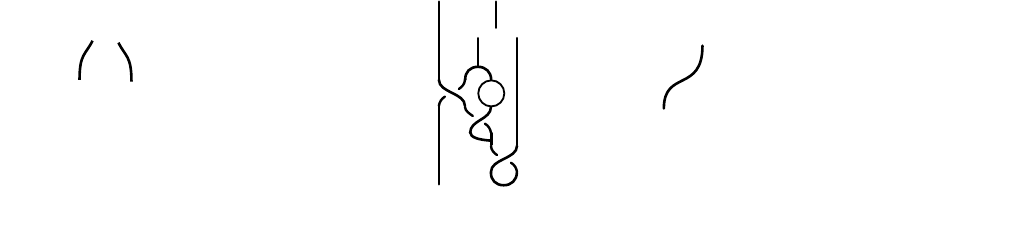}}.
\end{align}
Here, the first equality uses that $\Delta$ preserves the unit, which acts by the identity. The second equality uses \eqref{commutation-cond}, while the third equality uses naturality of the braiding. Finally, the last equality uses the unit axioms.
Pre-composing with $\Psi_{A \rtimes H, A \rtimes H}$, we see that Equation~\ref{eq4.9} is equivalent to 
\begin{align}\label{eq:center2}
m_{A\rtimes H}(\varphi \gamma \otimes \ide_A\otimes \; u_H)=m_{A\rtimes H}\Psi_{A\rtimes H,A\rtimes H}(\varphi \gamma\otimes \ide_A\otimes \; u_H).
\end{align}
This is Diagram~\ref{diag:leftcentzer} for the left centralizer of $A$ inside of $A\rtimes H$.
Hence, $\Gamma(C,\gamma)$ is an object in $\mathcal{C}^l_{A\rtimes H}(A)$. Conversely, if $(C',\zeta)$ is an object in $\mathcal{C}^l_{A\rtimes H}(A)$, then $(C',\varphi^{-1}\zeta)$ defines an object in $\mathcal{C}^l(\rR_\cB(A))$ such that $\Gamma(C',\varphi^{-1}\zeta)=(C',\zeta)$. This holds because after applying $\varphi^{-1}$ to the computation in \eqref{calc-center}, we recover \eqref{commutation-cond}. Hence $\Gamma$ is essentially surjective. As $\Gamma$ is the identity on morphisms, it is fully faithful and therefore an equivalence of categories by Lemma~\ref{cent-lemma}.

Next, we show that $\Gamma$ is a $\Psi^{-1}$-opposite monoidal functor. Indeed, given two objects $(C,\gamma)$ and $(C',\gamma')$ we have the following string of equalities of morphisms from $C\otimes C'$ to $A\rtimes H$:
\begin{align*}
\vcenter{\hbox{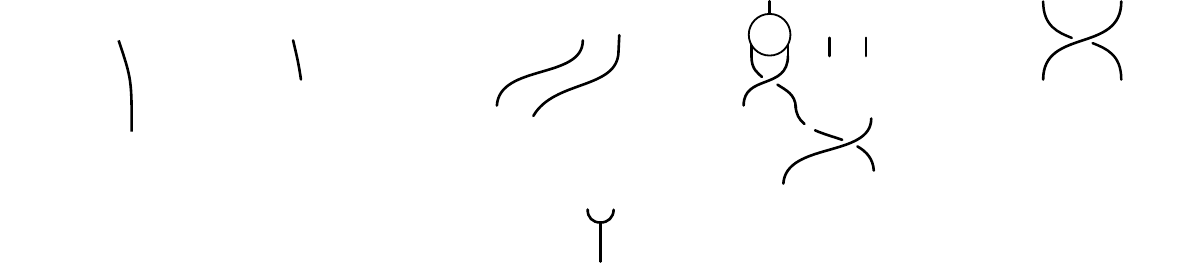}}.
\end{align*}
The first equality holds by \eqref{commutation-cond}. The second and fourth equality use naturality of the braiding, while the third equality uses the antipode axioms.
This calculation gives that 
$$\Psi^{-1}\colon \Gamma(C\otimes C')\rightarrow \Gamma(C')
\otimes \Gamma(C)$$
is an isomorphism in $\mathcal{C}^l_{A\rtimes H}(A)$.

Finally, the equivalence $\Gamma$ sends terminal objects to terminal objects. 
Using the discussion from Section \ref{sec:terminal}, this implies that there is an isomorphism of algebras 
\begin{equation} \label{eq:phiC}
\tilde{\varphi}\colon C^l(\rR_\cB(A))\to \Cent_{A\rtimes H}^l(A)^{\Psi^{-1}}
\end{equation}
such that the diagram 
\begin{align}\label{diag:phigamma}
\xymatrix{
C^l(\rR_\cB(A))\ar[rr]^{\tilde{\varphi}}\ar[d]_{\gamma_{\rR_\cB(A)}}&& \Cent_{A\rtimes H}^l(A)\ar[d]^{\gamma_{\text{Cent}_{A\rtimes H}^l(A)}}\\
\rR_\cB(A)\ar[rr]^{\varphi}&&A\rtimes H
}
\end{align}
of morphisms in $\cB$ commutes.
\end{proof}

\subsection{Consequences} \label{sec:cent-conseq}

\begin{corollary}\label{cor:YDalgebras}
The $H$-Yetter--Drinfeld structure on $A \rtimes H$ from Lemma \ref{AH-YD} 
makes $\Cent_{A\rtimes H}^l(A)^{\Psi^{-1}}$ an $H$-Yetter--Drinfeld module algebra such that $\tilde{\varphi}$ from \eqref{eq:phiC} is an isomorphism of commutative algebras in $\lYD{H}(\cB)$.
\end{corollary}

\begin{proof}
Denote $C_1:=C^l(\rR_\cB(A))$, with multiplication $m_1$. By Definition \ref{def:leftcenter}, $C_1$ is a subalgebra of $\rR_\cB(A)$ in $\lYD{H}(\cB)$. We denote its $H$-action and coaction by $a^{C_1},\delta^{C_1}$. Using the isomorphism $\tilde{\varphi}$ from Theorem \ref{cent-thm}, we define an $H$-action $a^{C_2}$ and an $H$-coaction $\delta^{C_2}$  on $C_2:=\Cent_{A\rtimes H}^l(A)$ by
$$a^{C_2}:=\tilde{\varphi}a^{C_1}(\ide_H\otimes \tilde{\varphi}^{-1}), \qquad \delta^{C_2}:=(\ide_H\otimes \tilde{\varphi})\delta^{C_1}\tilde{\varphi}^{-1}.$$
Now, $C_2\in \Alg(\lYD{H}(\cB))$ and $\tilde{\varphi}$ is a morphism of algebras in $\lYD{H}(\cB)$.
Further, $\gamma\colon C_2\to A\rtimes H$ becomes a morphism of Yetter--Drinfeld modules. The compatibility with $H$-coactions follows from
\begin{align*} 
(\ide_H\otimes \gamma)\delta^{C_2}=(\ide_H\otimes \gamma\tilde{\varphi})\delta^{C_1}\tilde{\varphi}^{-1}\stackrel{\text{\eqref{diag:phigamma}}}{=}&(\ide_H\otimes \varphi\gamma)\delta^{C_1}\tilde{\varphi}^{-1}=(\ide_H\otimes \varphi)\delta^{\rR}\gamma\tilde{\varphi}^{-1}\stackrel{\text{[Lemma~\ref{AH-YD}]}}{=}\delta^{\rtimes}\varphi\gamma\tilde{\varphi}^{-1}\\\stackrel{\text{\eqref{diag:phigamma}}}{=}&\delta^{\rtimes}\gamma\tilde{\varphi}\tilde{\varphi}^{-1}=\delta^{\rtimes}\gamma,
\end{align*}
and the compatibility with the $H$-actions is proved similarly.
\end{proof}

The result below also follows from Theorem~\ref{cent-thm}, cf. Equation \ref{eq:center2}.

\begin{corollary}
Assume that $\cB=\lmod{K}$ for $K$  a quasi-triangular Hopf algebra with braiding~$\Psi$. Then,  $Z_\cB(A)$ is the subalgebra $\Cent^l_{A\rtimes H}(A)^{\Psi^{-1}}$ of $(A\rtimes H)^{\Psi^{-1}}$, which is a $K$-module algebra, given~by 
\begin{align*} \textstyle
\left\lbrace \sum_i a_i\otimes h_i~~\middle|~~ m_{A\rtimes H}\Psi_{A\rtimes H,A\rtimes H}\left(\sum_i a_i\otimes h_i\otimes b \otimes 1_H\right)=m_{A\rtimes H} \left(\sum_i a_i\otimes h_i\otimes b \otimes 1_H\right), \quad \forall b\in A\right\rbrace.
\end{align*}
\end{corollary}

\smallskip

\begin{example}\label{expl:H-center}
Continuing Example \ref{adjoint-expl}, take $A=I_\cB$ the tensor unit of $\cB$ which is an algebra in $\lmod{H}(\cB)$. Then $Z_\cB(I)$ consists of all elements in $H$ that satisfy Equation \ref{commutation-cond} with $A = I_\cB$. 
This implies that $Z_\cB(I)=H$, which  provides a braided version of the fact that $H$ is a commutative algebra  in $\lYD{H}(\cB)$ when using the adjoint action. See, e.g., \cite{CFM}*{page~1332} when $\cB = \Vect$.
\end{example}

\begin{remark} \label{rem:cent} 
Recall that for an algebra $A$ in $\Vect$, the center $Z(A)$ of $A$ serves as a Morita invariant. Moreover, for an algebra $A$ in a ($\Vect$-central) monoidal category $\cC$, Davydov's full center $Z(A)$ of $A$ serves as a Morita invariant \cite{Dav1}*{Theorem~6.2, Corollary~6.3}. In particular, when $\cC = \lmod{H} = \lmod{H}(\Vect)$, for some Hopf algebra $H$ in $\Vect$, we have that $Z(A)$  is the centralizer algebra $\Cent_{A \rtimes H}(A)^{\oop}$ in $\cZ(\cC)$; see, e.g., \cite{Dav2}*{Theorem~5.3}. \footnote{After checking the results of \cite{Dav2}*{Theorem~5.3}, we added ${}^{\oop}$ to $\Cent_{A \rtimes H}(A)$.}

The main results of this manuscript, Theorems~\ref{thm:Morita} and~\ref{cent-thm}, give  generalizations of the results above. Namely, for an algebra $A$ in a $\cB$-central monoidal category $\cC$, the $\cB$-center $Z_\cB(A)$ of $A$ serves as a Morita invariant, and in the case when $\cC = \lmod{H}(\cB)$, for some Hopf algebra $H$ in $\cB$, we have that $Z_\cB(A)$  is the braided centralizer algebra $\Cent^l_{A \rtimes H}(A)^{\Psi^{-1}}$ in $\cZ_\cB(\cC)$.
\end{remark}


\medskip

\section{Braided Drinfeld doubles and Heisenberg doubles}\label{sec:doubles}

Towards obtaining concrete examples of the results in the previous sections, we discuss here braided versions of useful algebraic constructions: the Drinfeld double and the Heisenberg double. Here, we restrict our attention to the case when $\cC = \lmod{H}(\cB)$ as in Example~\ref{ex:augment}(3). Consider the following notation that we will use below and in the following sections.

\begin{notation} \label{not:Sweedler} Here, $\cB = \lmod{K}$ for $K$ a quasi-triangular Hopf algebra over $\Bbbk$  with $\cal{R}$-matrix and its inverse denoted $R^{(1)} \otimes R^{(2)}$ and $R^{(-1)} \otimes R^{(-2)}$, respectively. The braiding $\Psi_{V,W}$ for objects $V,W$ in $\cB$ is given by
$$\Psi_{V,W}(v \otimes w) = (R^{(2)} \cdot w) \otimes (R^{(1)} \cdot v), \quad \forall v\in V,w\in W.$$
Take $H$ to be a Hopf algebra in $\cB$. We use sumless Sweedler notation
$\Delta(b):= b_{(1)}  \otimes  b_{(2)}$ and $\Delta(d):= d_{(1)} \otimes d_{(2)}$ for $b \in H$, $d \in K$, and the coaction $\delta \colon V \to H \otimes V$  for $V$ in $\lYD{H}(\cB)$ is denoted by $\delta(v)= v_{(-1)} \otimes v_{(0)}$.
\end{notation}

\smallskip

We now recall material about {\it braided Drinfeld doubles} from \cite{L17}; this construction is due to  \cite{Maj99} where it is called {\it double bosonization}.

\begin{definition} \label{def:DrinKH} \cite{L17}*{Definition~3.5} 
Take a Hopf algebra $H$ in $\lmod{K}$, along with dual Hopf algebra $H^*$ in $\lmod{K}$, and a nondegenerate Hopf algebra pairing $$\langle \;, \rangle\colon  H^* \otimes H \to \Bbbk,$$
i.e. the left and right radical of $\langle \;, \rangle$ are trivial. Then, the 
{\it braided Drinfeld double $\textnormal{Drin}_K(H^*,H)$ of $H$ with respect to $K$ and $\langle \;, \rangle$} is defined to be the Hopf algebra over $\Bbbk$ that is $H^* \otimes K \otimes H$ as a $\Bbbk$-vector space, and for $b \in H$, $c \in H^*$, $d \in K$, has multiplication
$$ db = (d_{(1)} \cdot b)\; d_{(2)}, \quad \quad
dc = (d_{(1)} \cdot c)\; d_{(2)}$$
$$(R^{(-1)} \cdot b_{(2)}) \; (R^{(-2)} \cdot c_{(1)}) \; \langle c_{(2)},~b_{(1)} \rangle ~=~
R^{(-1)} \; c_{(2)} \; b_{(1)} \; R^{(2)} \; \langle R^{-(2)} \cdot c_{(1)},~R^{(1)} \cdot b_{(2)} \rangle, $$
with coproduct 
$$\Delta(d) = d_{(1)} \otimes d_{(2)}, \quad
\Delta(b) = b_{(1)}\; R^{(2)}  \otimes   (R^{(1)} \cdot b_{(2)}), 
\quad
\Delta(c) =  R^{-(1)} \; c_{(2)}  \otimes   (R^{-(2)} \cdot c_{(1)}),$$
with counit the same as on $H^*$, $K$, $H$ and extended multiplicatively, and with antipode
$$S(d) = S_K(d), \quad \quad S(b) = S_K(R^{(2)})\;(R^{(1)} \cdot S_H(b)), \quad \quad S(c) = S_K(R^{-(1)})\;(R^{-(2)} \cdot S^{-1}_{H^*}(c)).$$

\end{definition}

\begin{proposition} \label{prop:Phi} \cite{L17}*{Proposition~3.6} There is a fully faithful functor of monoidal categories:
$$\Phi\colon \lYD{H}(\lmod{K}) {\rightarrow} \lmod{\textnormal{Drin}_K(H^*,H)} $$
assigning $\Phi(V) = V$ with the same $K$- and $H$-action, and with $H^*$-action defined by 
$$c \cdot v = \langle c, v_{(-1)} \rangle \;v_{(0)},$$
for all $c \in H^*$ and $v \in V$. 

If $H$ is finite-dimensional, then $\Phi$ is an equivalence of braided monoidal categories.
\qed
\end{proposition}

\begin{example}
If $H = \Bbbk$, then $\textnormal{Drin}_K(\Bbbk^*, \Bbbk) \cong K$ as $\Bbbk$-Hopf algebras. Moreover, $\lmod{\Bbbk}(\cB)$ is isomorphic to $\cB$ as braided monoidal categories. So, Propositions~\ref{prop:Phi} and~\ref{prop:YD} recover the fact that $\cZ_\cB(\cB)$ is isomorphic to $\cB$ as braided monoidal categories as stated in Proposition~\ref{prop:rel-center}(1). 

When $K=\Bbbk$, the braided Drinfeld double is the usual Drinfeld double or \emph{quantum double} as found, for example, in \cite{Majid}*{Theorem 7.1.1}.
\end{example}

We saw in the previous section that the $\cB$-center of $A \in \Alg(\lmod{H}(\cB))$ is isomorphic to the  braided centralizer algebra of $A$ in $A \rtimes H$. So by continuing Example~\ref{coreg-expl} and using the braided smash product algebras discussed in Section~\ref{sec:smash}, consider the following special subclass of such algebras $A \rtimes H$.

\begin{definition} \label{def:braidHeis} \cite{L17}*{Example~3.10}
For $H$ a Hopf algebra in a braided monoidal category $(\cB,\Psi)$, the braided smash product algebra $\rdual{H}^{\Psi^{-1}}\rtimes H$ is called the \emph{braided Heisenberg double} of $H$, and is denoted by $\Heis_\cB(H,\rdual{H})$.
\end{definition}

Thus, $\Heis_\cB(H,\rdual{H})$ is the $K$-module $\rdual{H}\otimes H$ with multiplication  $m_{\Heis}$ given by 
$$
m_{\Heis}(a\otimes g\otimes b\otimes h)~=~m_{\rdual{H}}(R^{(-1)}R'^{(2)}\cdot b_{(2)}\otimes R^{(-2)}\cdot a)\otimes m_H(R''^{(1)}R'^{(1)}\cdot g_{(2)}\otimes h)\langle g_{(1)},R''^{(2)}\cdot b_{(1)} \rangle,
$$

\noindent for $a,b\in \rdual{H}$ and $g,h\in H$. Here $\langle~,~\rangle\colon H\otimes \rdual{H}\to \Bbbk$ is a nondegenerate Hopf algebra pairing, and $\cal{R}'$, $\cal{R}''$ are copies of the universal $\cal{R}$-matrix of $K$.
If $H$ is a $\Bbbk$-Hopf algebra, then one recovers the Heisenberg double $\Heis(H)$ from \cite{Lu2}*{Definition~5.1}.

\smallskip

We end with a basic example of a (braided) Heisenberg double for $\cB = \Vect$ and compute an example of a $\Vect$- (or full-) center of an algebra in $\lmod{H}(\Vect)$ below.

\begin{example}\label{expl:classic}
Let $H=\Bbbk[x_1,\ldots, x_n]$ and write $\rdual{H}=\Bbbk[\partial_1,\ldots, \partial_n]$, with pairing $\langle x_i,\partial_j\rangle=\delta_{i,j}$. This gives dually paired Hopf algebras over $\Bbbk$, and $$\Heis_{\Vect}(H,\rdual{H}) = \frac{\Bbbk\langle x_1,\ldots,x_n,\partial_1,\ldots,\partial_n\rangle}{\left([x_i,\partial_j]-\delta_{i,j}, ~~[x_i,x_j], ~~[\partial_i,\partial_j]\right)} =: A_n(\Bbbk),$$
the \emph{$n$-th Weyl algebra} (also called \emph{Heisenberg algebra}). 
An elementary computation  shows that $$Z_{\Vect}(\rdual{H})=\Cent^l_{A_n(\Bbbk)}(\rdual{H})=\Bbbk[\partial_1,\ldots,\partial_n]=\rdual{H}.$$
\end{example}


\medskip

\section{Example: Module algebras over \texorpdfstring{$u_q(\mathfrak{sl}_2)$}{uq(sl2)}}\label{sec:uqsl2}

In this section we provide an extended example of the material in the previous sections for the representation category of the finite-dimensional small quantum group $u_q(\mathfrak{sl}_2)$. Recall Notation~\ref{not:Sweedler} and consider the following notation for the rest of this section.

\begin{notation} \label{not:uqsl2}
Let $n\geq 3$ be an integer , and let $q$ a root of unity so that $q^2$ has order $n$. 
\begin{itemize}[leftmargin=*]
\item From \cite{Majid}*{Lemma~2.1.2 and Example~2.1.6}, let $K$ be the quasi-triangular Hopf algebra $\Bbbk \mathbb{Z}_n$ where $\mathbb{Z}_n = \langle g ~|~ g^n =1 \rangle$, with $\cal{R}$-matrix and inverse given by $$\cal{R} = \frac{1}{n} \sum_{i,j = 0}^{n-1} q^{-2ij}g^i \otimes g^j \quad \text{and} \quad \cal{R}^{-1} = \frac{1}{n} \sum_{i,j = 0}^{n-1} q^{-2ij}g^{-i} \otimes g^j.$$

\item Let $\cB$ be the braided monoidal category $\lmod{K}$, which then, for $g \cdot v = q^{2|v|}v$ with $v$ belonging to an object in $\lmod{K}$, has braiding and inverse braiding $$\quad \Psi(v \otimes w) = \frac{1}{n} \sum_{i,j = 0}^{n-1} q^{-2ij}(g^j \cdot w) \otimes (g^i \cdot v) = q^{2|v||w|}w \otimes v
\quad \text{and} \quad \Psi^{-1}(v \otimes w) = q^{-2|v||w|}w \otimes v.$$

\item Take $H$ to be the Hopf algebra $\Bbbk[x]/(x^n)$ in $\cB$, where $$\Delta(x^m) = \sum_{i=0}^m ~{m\choose i}_{q^2} x^i \otimes x^{m-i} \quad \text{ for } \quad {m\choose i}_{q^2} = \prod_{j=0}^{i-1} \frac{1-q^{2(i-j)}}{1-q^{2(j+1)}},$$
along with $\varepsilon(x^m) = \delta_{m,0}$ and $S(x^m) = (-1)^m q^{2{m \choose 2}}x^m$. Here, the $K$-action and the induced $K$-coaction on $H$ are given by 
$$g \cdot x = q^{-2} x \quad \quad  \text{and} \quad \quad  \delta(x) =  \sum_{i,j = 0}^{n-1} R^{(2)} \otimes (R^{(1)} \cdot x) = g^{-1} \otimes x.$$

\smallskip 

\item Next, take $\cC$ to be the monoidal category $\lmod{H}{(\cB)}$, which is equivalent to $\lmod{(H \rtimes K)}$. The smash product algebra $H \rtimes K$ is the \emph{Taft algebra} $T_n(q^{-2})$, i.e., the $\Bbbk$-Hopf algebra 
$$T_n(q^{-2}) =  \Bbbk \langle g,x \rangle / (g^n-1, ~x^n, ~gx - q^{-2} xg),$$
with $\Delta(g) = g \otimes g, ~\Delta(x) = g^{-1} \otimes x + x \otimes 1, ~\varepsilon(g) = 1, ~\varepsilon(x) =0, ~S(g) = g^{-1}, ~S(x) = -gx$. That is, $\cC$ is equivalent to $\lmod{T_n(q^{-2})}$ as a monoidal category.

\smallskip

\item Pick $A:= A_\gamma = \Bbbk[u] \in \Alg{(\cC)}$ with $$g \cdot u = q^{2}u \quad \text{ and } \quad x \cdot u = \gamma 1_A, \text{ for } \gamma  \in \Bbbk.$$
\end{itemize}
\end{notation}

\begin{remark} \label{rem:gens} We will show in Corollary~\ref{cor:MoritaSec6} below that $A_{\gamma =0}$ and $A_{\gamma \neq 0}$ are not Morita equivalent as algebras in $\cC$, and to do so we compute their respective $\cB$-centers and employ Theorem~\ref{thm:Morita}. Even though this computation is tedious, it is, in a sense, more efficient to use the $\cB$-center $Z_\cB(A) = C^l(\rR_\cB(A))$ [Theorem~\ref{thm:Bcenter}] as a Morita invariant rather than the full center $Z(A) = C^l(\rR(A))$ \cite{Dav1}*{Theorem~5.4}: Indeed, $\rR(A) = H \otimes K \otimes A$  as a $\Bbbk$-vector space \cite{Dav2}*{Proposition~5.1}, whereas $\rR_\cB(A) = H \otimes A$ as a $\Bbbk$-vector space [Theorem~\ref{prop:RB-YD}] and is a smaller algebra on which to do computations.
\end{remark}

Recall that the functor $\rR_\cB\colon \cC \to \cZ_{\cB}(\cC)$ exists in the setting above by Theorem~\ref{prop:RB-YD} and it is lax monoidal. We compute the algebra $\rR_\cB(A)$ below.

\begin{lemma} \label{lem:uqsl2RBA} Retain the notation above. 
Then the algebra $\rR_\cB(A)$ in $\lYD{H}(\cB)$  has the $\Bbbk$-algebra presentation
$$\Bbbk \langle {\bf y}  , {\bf u} \rangle/({\bf y}^n,~ {\bf u}{\bf y} - q^{-2}{\bf y}{\bf u}).$$ 
As an object in $\cB$, the $K$-action on $\rR_\cB(A)$ is given by 
$$a(g \otimes  {\bf y})  = q^{-2} {\bf y}  \quad \text{ and } \quad 
a(g \otimes {\bf u})  = q^{2} {\bf u}.$$ 
For the Yetter--Drinfeld structure, the $H$-action and $H$-coaction on $\rR_\cB(A)$ are given by
$$a^{\rR}(x \otimes {\bf y}) = (1-q^2) {\bf y}^2, 
\; \; a^{\rR}(x \otimes {\bf u}) = (1-q^{-4}) {\bf y}{\bf u} + \gamma {\bf 1}, 
\; \; \delta^{\rR}({\bf y}) = 1_H \otimes {\bf y} + {x} \otimes 1_A, 
\; \; \delta^{\rR}({\bf u}) = 1_H \otimes {\bf u}.$$
\end{lemma}

\begin{proof}
As an object in $\lYD{H}{(\cB)}$, we have that $\rR_\cB(A) = H \otimes A$. Here, we take $H$ to be $\Bbbk[y]/(y^n)$ and denote the generators of $H$ and of $A$ in $\rR_\cB(A)$ by $${\bf y}:=y \otimes 1_A \quad \text{and} \quad {\bf u}:=1_H \otimes  u,$$ respectively. Moreover, take ${\bf 1}:= 1_H  \otimes 1_A$. The multiplication of $\rR_\cB(A)$  from Corollary~\ref{cor:RBA} yields the relation ${\bf u}{\bf y} = q^{-2}{\bf y}{\bf u}$ in $\rR_\cB(A)$; its other relations come from~$H$.

\smallskip
The $K$-action $a$ on $\rR_\cB(A)$ is the $K$-action on $H$ and on $A$ induced by the set-up of Notation~\ref{not:uqsl2}.

\smallskip The Yetter--Drinfeld structure $(a^{\rR}, \delta^{\rR})$ on $\rR_\cB(A)$ is given in Theorem~\ref{prop:RB-YD}. We provide the details of one computation and leave the rest, including the verification of the Yetter--Drinfeld compatibility condition [Definition~\ref{def:HYD(B)}], to the reader: 
\begingroup
\allowdisplaybreaks
\begin{align*}
\vspace{.05in}
a^{\rR}(x \otimes 1_H \otimes u)&=
[(m_H \otimes \ide_A)(\ide_H \otimes \Psi_{A,H})(\ide_H \otimes \; a_A \otimes S_H)(\ide_{H \otimes H} \otimes \Psi_{H,A})\\
\vspace{.15in}
&\quad \circ(m_H \otimes \Delta_H \otimes \ide_A)(\ide_H \otimes \Psi_{H,H} \otimes \ide_A)(\Delta \otimes\ide_{H \otimes A})](y \; \otimes \; 1_H \; \otimes \; u)
\\
\vspace{.05in}
&=
[(m_H \otimes \ide_A)(\ide_H \otimes \Psi_{A,H})(\ide_H \otimes \; a_A \otimes S_H)(\ide_{H \otimes H} \otimes \Psi_{H,A})\\
\vspace{.15in}
&\quad \circ(m_H \otimes \Delta_H \otimes \ide_A)(\ide_H \otimes \Psi_{H,H} \otimes \ide_A)](1_H \otimes y \otimes 1_H \otimes u + y \otimes 1_H \otimes  1_H \otimes u)
\\
\vspace{.05in}
&=
[(m_H \otimes \ide_A)(\ide_H \otimes \Psi_{A,H})(\ide_H \otimes \; a_A \otimes S_H)(\ide_{H \otimes H} \otimes \Psi_{H,A})\\
\vspace{.15in}
&\quad \circ(m_H \otimes \Delta_H \otimes \ide_A)]
(1_H   \otimes  1_H  \otimes  y  \otimes  u + y   \otimes  1_H \otimes 1_H \otimes u)
\\
\vspace{.05in}
&=
[(m_H \otimes\ide_A)(\ide_H \otimes \Psi_{A,H})(\ide_H \otimes \; a_A \otimes S_H)(\ide_{H \otimes H} \otimes \Psi_{H,A})]\\
\vspace{.15in}
&\quad \quad
(1_H  \otimes  1_H \otimes y \otimes u 
+ 
1_H  \otimes y \otimes 1_H \otimes u 
+ 
y  \otimes 1_H \otimes 1_H \otimes u)
\\
\vspace{.05in}
&=
[(m_H \otimes \ide_A)(\ide_H \otimes \Psi_{A,H})(\ide_H \otimes \; a_A \otimes S_H)]\\
\vspace{.15in}
&\quad \quad
(q^{-2}(1_H  \otimes 1_H \otimes u \otimes y) 
+ 
(1_H  \otimes y \otimes u \otimes 1_H) 
+ 
(y  \otimes 1_H \otimes u \otimes 1_H))
\\
\vspace{.05in}
&=
[(m_H \otimes \ide_A)(\ide_H \otimes \Psi_{A,H})]\\
\vspace{.15in}
&\quad \quad
(-q^{-2}(1_H \otimes u \otimes y) 
+ 
\gamma(1_H  \otimes 1_A \otimes 1_H) 
+ 
(y  \otimes u \otimes 1_H))
\\
\vspace{.15in}
&=
[(m_H \otimes \ide_A)]
(-q^{-4}(1_H \otimes y \otimes  u) 
+ 
\gamma(1_H \otimes 1_H \otimes 1_A) 
+ 
(y  \otimes 1_H \otimes u))
\\
&=
-q^{-4}(y \otimes u) 
+ 
\gamma(1_H \otimes 1_A) 
+ 
(y \otimes u).
\end{align*}
\endgroup

\vspace{-.25in}
\end{proof}

\begin{proposition} \label{prop:uqsl2RBA} Retain the notation above.
Then, the $\cB$-center of $A=A_\gamma$ is 
$$Z_\cB(A_\gamma) = \begin{cases}
H \otimes \Bbbk[{\bf u}^n], & \text{ for } \gamma = 0\\
\Bbbk[{\bf z}], & \text{ for } \gamma \neq 0,
\end{cases}
\quad \text{with} \quad {\bf z}:= \sum_{i=0}^{n-1} \gamma^{-i} q^{-2({i+1 \choose 2}+i)} (1-q^2)^i  ({\bf y}^i \otimes {\bf u}^{i+1}).$$
For $\gamma \neq 0$, we have that as an object in $\lYD{H}(\cB)$, the braided commutative algebra $Z_\cB(A_{\gamma \neq 0})$ has $K$-action, $H$-action, and $H$-coaction given by
$$a(g \otimes  {\bf z}) = q^{2} {\bf z}, \quad \quad a^{\rR}(x \otimes {\bf z}) = \gamma {\bf 1}, \quad \quad \delta^{\rR}({\bf z}) = \sum_{i=0}^{n-1} \gamma^{-i} (1-q^2)^i q^{-2({i+1 \choose 2}+i)} ({\bf y}^i \otimes  {\bf z}^{i+1}).$$
\end{proposition}

\begin{proof}
Using Theorem~\ref{thm:Bcenter}, we compute the $\cB$-center of $A$ by computing the left center of the algebra $\rR_\cB(A)$ given in Lemma~\ref{lem:uqsl2RBA}. By Examples~\ref{ex:ClA} and~\ref{ex:CentlA}, we have that
$$C^l\rR_\cB(A) = \left\{{\bf s}:=\sum_{i=0}^{n-1} \sum_{j \geq 0} \lambda_{i,j} \; {\bf y}^i \otimes {\bf u}^j \in H \otimes A  ~ \Big|~ m_{H \otimes A}({\bf s} \otimes {\bf t}) = m_{H \otimes A}\; \Psi^{\bf YD}\;({\bf s} \otimes {\bf t}), ~\forall {\bf t} \in H \otimes A\right\},$$
for $\lambda_{i,j} \in \Bbbk$, and $\Psi^{\bf YD}$ from Definition~\ref{def:HYD(B)} and Theorem~\ref{prop:RB-YD}:
$$\Psi^{\bf YD}~:= \Psi^{\lYD{H}(\cB)}_{H \otimes A, ~H \otimes A} ~= (a^{\rR}_{H \otimes A} \otimes \text{Id}_{H \otimes A})\;(\text{Id}_H \otimes \Psi^{\cB}_{H \otimes A, ~H \otimes A})\;(\delta^{\rR}_{H \otimes A} \otimes \text{Id}_{H \otimes A}).$$

\smallskip

One can check that choosing ${\bf t} = {\bf y}$ does not yield conditions on ${\bf s}$. On the other hand, choosing ${\bf t} = {\bf u}$ yields the following conditions on ${\bf s}$: 
\begin{equation} \label{eq:lambda}
(q^{2(j-1)} - 1) \;\lambda_{i,j-1} + \gamma\; \frac{q^{2j}(1-q^{2(i+1)})}{1-q^2} \; \lambda_{i+1, j} = 0, \quad \forall i,j \geq 0.
\end{equation}

If $\gamma =0$, then $\lambda_{i,j-1}$ is free for $j \equiv 1\mod n$, and equal to 0 otherwise. Then
\begin{equation} \label{eq:gamma0}
C^l\rR_\cB(A) = \left\{{\bf s}:=\sum_{i=0}^{n-1} \sum_{k \geq 0} \lambda_{i,kn} \; {\bf y}^i \otimes  {\bf u}^{kn}\right\} = H \otimes \Bbbk[{\bf u}^n] = Z_\cB(A).
\end{equation}

If $\gamma \neq 0$, then taking 0 to be in the second slot of $\lambda_{-,-}$ in \eqref{eq:lambda} implies only that ${\bf 1} \in Z_\cB(A)$. Taking 0 to be in the first slot of $\lambda_{-,-}$ in \eqref{eq:lambda} implies that the elements
$${\bf z}_\ell:= \sum_{k=0}^{n-1} \gamma^{-k} {\ell + k -1 \choose k}_{q^2}\; q^{-2(k\ell + {k+1 \choose 2})}\; (1-q^2)^k \; ({\bf y}^k \otimes {\bf u}^{k + \ell}),~~\forall \ell \geq 0$$
are in $Z_\cB(A)$. Moreover, one can check that 
$${\bf z}_\ell = ({\bf z}_1)^\ell, ~~\forall \ell \geq 0.$$
So taking ${\bf z}:= {\bf z}_1$, we get
\begin{equation} \label{eq:gammaneq0}
C^l\rR_\cB(A) = \left\{{\bf s}:= \sum_{k \geq 0} \lambda_{k} \; {\bf z}^{k}\right\} =  \Bbbk[{\bf z}] = Z_\cB(A).
\end{equation}

Now Equations \ref{eq:gamma0} and \ref{eq:gammaneq0} give us the description $Z_\cB(A_\gamma)$.

\smallskip

Determining the Yetter--Drinfeld structure on $Z_\cB(A_{\gamma \neq 0})$ is a routine calculation using the action and coaction from Theorem~\ref{prop:RB-YD} and compatibility condition from Definition~\ref{def:HYD(B)}.
\end{proof}

Now with Theorem~\ref{thm:Morita} we arrive at the consequence below.

\begin{corollary} \label{cor:MoritaSec6}
The objects $A_{\gamma = 0}$ and $A_{\gamma \neq 0}$ are  Morita inequivalent as algebras in $\cC$. \qed
\end{corollary}

Finally, we translate these results to $\lmod{u_q(\mathfrak{sl}_2)}$. To do so,
we return to braided Drinfeld doubles from Definition~\ref{def:DrinKH}.

\begin{lemma}  \label{lem:DKH}
Consider the quasi-triangular $\Bbbk$-Hopf algebra $K = \Bbbk \mathbb{Z}_n$, with $\mathbb{Z}_n  = \langle g ~|~ g^n=1\rangle$, and take Hopf algebras $H = \Bbbk[x]/(x^n)$ and $H^* = \Bbbk[x^*]/((x^*)^n)$  in $\lmod{K}$ as in Notation~\ref{not:uqsl2}. Choose the nondegenerate pairing $\langle ~, \rangle: H^* \otimes H \to \Bbbk$ determined by
$$\langle x^*, x \rangle = \frac{1}{q - q^{-1}}.$$
Then, the braided Drinfeld double $\textnormal{Drin}_K(H^*,H)$ of $H$ with respect to $K$ and the pairing above is generated by a group-like element $g$, a $(g^{-1},1)$-skew primitive element $x$, and a $(g^{-1},1)$-skew primitive element $x^*$, subject to relations:
$$g^n = 1, \quad x^n = (x^*)^n =0, \quad gx = q^{-2} xg, \quad gx^* = q^{2}x^*g, \quad x^* x - q^2 x x^* = \frac{1}{q-q^{-1}}(1-g^{-2}).$$
\end{lemma}

\begin{proof}
We provide one computation and leave the rest to the reader. Note that
$$R^{(2)} \otimes (R^{(1)} \cdot x) = g^{-1} \otimes x \quad \quad \text{and} \quad
\quad R^{-(1)} \otimes (R^{-(2)} \cdot x^*) = g^{-1} \otimes x^*.$$
Now, from Definition~\ref{def:DrinKH}, consider the relation $$(R^{(-1)} \cdot b_{(2)}) \; (R^{(-2)} \cdot c_{(1)}) \; \langle c_{(2)},~b_{(1)} \rangle ~=~
R^{(-1)} \; c_{(2)} \; b_{(1)}) \; R^{(2)} \; \langle R^{-(2)} \cdot c_{(1)},~R^{(1)} \cdot b_{(2)} \rangle, $$
for $b = x$ and $c = x^*$. 
The left-hand side is
\[
\begin{array}{rl}
(R^{(-1)} \cdot x) \; (R^{(-2)} \cdot x^*) \; \langle 1^*,~1 \rangle + (R^{(-1)} \cdot 1) \; (R^{(-2)} \cdot 1^*) \; \langle x^*,~x \rangle
&~=~ (g^{-1} \cdot x)x^* + (g^{-1} \cdot 1) 1^* \langle x^*, x \rangle\\
&~=~ q^2 xx^* + \frac{1}{q - q^{-1}},
\end{array}
\]
and the right-side is
\[
\begin{array}{l}
\smallskip
R^{(-1)} \; x^* \;x \; R^{(2)} \; \langle R^{-(2)} \cdot 1^*,~R^{(1)} \cdot 1 \rangle
+ R^{(-1)} \; 1^* \;1 \; R^{(2)} \; \langle R^{-(2)} \cdot x^*,~R^{(1)} \cdot x \rangle\\
\smallskip
\quad ~=~ R^{(-1)} \; x^* \;x \; R^{(2)} \; \langle \varepsilon(R^{-(2)}),~\varepsilon(R^{(1)}) \rangle
+ g^{-1} \; 1^* \;1 \; g^{-1} \; \langle x^*,~ x \rangle\\
\quad ~=~ R^{(-1)}\varepsilon(R^{-(2)}) \; x^* \;x \; \varepsilon(R^{(1)}) R^{(2)}  \; \langle 1,~1 \rangle
+ g^{-2} \frac{1}{q - q^{-1}}
\smallskip
\\
\quad ~=~  x^* \;x 
+ g^{-2} \frac{1}{q - q^{-1}}.
\end{array}
\]
The last equation holds as $(\text{Id} \otimes \varepsilon) \cal{R}^{-1} = (\varepsilon \otimes \text{id}) \cal{R} = 1$. Thus,  $x^* x - q^2 x x^* = \frac{1}{q-q^{-1}}(1-g^{-2})$.
\end{proof}

Now consider the small quantum group $u_q(\mathfrak{sl}_2)$, for $q$ a root of unity so that $q^2$ has order $n$ for $n \geq 3$ (e.g., as in \cite{Kassel}).  We take $u_q(\mathfrak{sl}_2)$ to be generated by indeterminates $k$, $e$, $f$, subject to relations
$$k^n =1, \quad e^n = f^n = 0, \quad ke = q^2 ek, \quad kf = q^{-2}fk, \quad ef-fe = \frac{k - k^{-1}}{q-q^{-1}},$$
where 
$$\Delta(k) = k \otimes k, \quad \Delta(e) = 1 \otimes e + e \otimes k,  \quad \Delta(f) = k^{-1} \otimes f + f \otimes 1, \quad \varepsilon(k) = 1,  \quad \varepsilon(e) = \varepsilon(f) =0.$$

\begin{proposition} \label{prop:uqaction}
\begin{enumerate}
\item[\textnormal{(1)}] We have that $u_q(\mathfrak{sl}_2)$ is isomorphic to the braided Drinfeld double $\textnormal{Drin}_K(H^*,H)$ from Lemma~\ref{lem:DKH}.

\smallskip

\item[\textnormal{(2)}] The Hopf subalgebra $u_q(\mathfrak{sl}^-_2)$ (the negative Borel part) of $u_q(\mathfrak{sl}_2)$ generated by $k$ and $f$  is isomorphic to the Taft algebra $T_n(q^{-2})$.

\smallskip

\item[\textnormal{(3)}] We have that $\rR_\cB(A_\gamma)$ from Lemma~\ref{lem:uqsl2RBA} is an algebra in $\lmod{u_q(\mathfrak{sl}_2)}$ via
\[
\begin{array}{lll}
k \cdot {\bf y} = q^{-2} {\bf y}, \quad
&f \cdot {\bf y} = (1-q^2) {\bf y}^2, \quad
&e \cdot {\bf y} = \frac{1}{q-q^{-1}},\\
k \cdot {\bf u} = q^{2} {\bf u}, \quad 
& f \cdot {\bf u} = (1-q^{-4}) {\bf y}{\bf u} +\gamma, \quad 
&e \cdot {\bf u} = 0.
\end{array}
\]

\smallskip

\item[\textnormal{(4)}] Moreover, $Z_\cB(A_\gamma)$ from Proposition~\ref{prop:uqsl2RBA} is a commutative algebra in $\lmod{u_q(\mathfrak{sl}_2)}$ via
$$ \textstyle k \cdot {\bf z} = q^{2} {\bf z}, \quad \quad f \cdot {\bf z} = \gamma,\quad \quad e \cdot {\bf z} = -q\gamma^{-1}{\bf z}^2.$$
\end{enumerate}
\end{proposition}

\begin{proof}
(1,2) The isomorphism $\pi$ from $\textnormal{Drin}_K(H^*,H)$ to $u_q(\mathfrak{sl}_2)$ is given by $g \mapsto k$, $x \mapsto f$,  $x^* \mapsto k^{-1}e$, which yields (1). Part (2) follows directly using this isomorphism $\pi$.

\smallskip

(3) The action of $k$ and of $f$ follows from Lemma~\ref{lem:uqsl2RBA} using the isomorphism $\pi$ from above. By Proposition~\ref{prop:Phi} and Lemma~\ref{lem:uqsl2RBA}, along with the pairing in Lemma~\ref{lem:DKH}, we have that
\[
\begin{array}{c}
\smallskip
x^* \cdot {\bf y}  ~= ~\langle x^*, {\bf y}_{(-1)}\rangle\; {\bf y}_{(0)}~ = ~\langle x^*, 1_H \rangle \;{\bf y} + \langle x^*, {\bf y} \rangle \;1_A  ~=~ \frac{1}{q-q^{-1}},\\
$$x^* \cdot {\bf u}  ~= ~\langle x^*, {\bf u}_{(-1)}\rangle\; {\bf u}_{(0)}~ = ~\langle x^*, 1_H \rangle \; {\bf u}  ~=~ 0.
\end{array}
\]
Now the conclusion holds by the isomorphism $\pi$ from above.

\smallskip

(4) The action of $k$ and of $f$ follows from Proposition~\ref{prop:uqsl2RBA} and using $\pi$ from above. By Propositions~\ref{prop:Phi} and~\ref{prop:uqsl2RBA}, along with the pairing in Lemma~\ref{lem:DKH}, we have that
$$\textstyle x^* \cdot {\bf z}  ~= ~\langle x^*, {\bf z}_{(-1)}\rangle\; {\bf z}_{(0)}~ = \gamma^{-1} (1-q^2)q^{-2(2)}\langle x^*, {\bf y} \rangle \; {\bf z}^2 ~=~ \frac{1-q^2}{\gamma q^4(q-q^{-1})}{\bf z}^2 ~=~ -\frac{1}{\gamma q^3}{\bf z}^2.$$
Namely, for $\delta^{\rR}({\bf z})$ in Proposition~\ref{prop:uqsl2RBA}, we get that $\langle x^*, {\bf y}^i \rangle$ is nonzero only when $i=1$. Now the conclusion holds by the isomorphism $\pi$ above.
\end{proof}


\medskip

\section{Example: Module algebras over the Sweedler Hopf algebra} \label{sec:T2}

We provide an example illustrating that the $\cB$-center of a parametrized family of algebras can be parameter-independent. Take char$(\Bbbk) \neq 2$, recall Notation~\ref{not:Sweedler}, and consider the following notation.

\begin{notation} \label{not:T2}
For $\xi \in \Bbbk$, let $(K, \cal{R}_{\xi})$ be the Sweedler Hopf algebra 
$$K = T_2(-1) = \Bbbk \langle g, x \rangle/(g^2-1, ~x^2, ~gx+xg) ,$$
with $g$ grouplike and $x$ being $(g,1)$-skew-primitive, which is quasi-triangular with $\cal{R}$-matrix
$$\textstyle \cal{R}_{\xi} = \frac{1}{2}(1 \otimes 1 + 1 \otimes g + g \otimes 1 - g \otimes g) + \frac{\xi}{2}(x \otimes x + x \otimes gx + gx \otimes gx - gx \otimes x),$$
for $\xi \in \Bbbk$.
Take $\cal{B} = \lmod{K}$, which has braiding $\Psi_{\xi}^\cal{B}(v,w) = (\cal{R}_{\xi}^{(2)} \cdot w) \otimes (\cal{R}_{\xi}^{(1)} \cdot v)$. Let $\cC = \lmod{H}(\cB)$, with $H = \Bbbk$ a Hopf algebra in $\cB$. Take $A_\gamma$ to be $ \Bbbk[u] \in \Alg(\cC)$ with 
$$g \cdot u = -u \quad \text{and} \quad x \cdot u = \gamma, \quad \text{for} \quad\gamma \in \Bbbk.$$
\end{notation}

As discussed in Section~\ref{sec:Cbraided}, both $\cC$ and its $\cB$-center $\cZ_\cB(\cC)$ are isomorphic to  $\cB$, and moreover $\Psi^{\cZ_\cB(\cC)} = \Psi^\cB =: \Psi$. Further,  the $\cB$-center $Z_\cB(A_\gamma)$  is actually the left center $C^l(A_\gamma)$. Now
$$ \textstyle C^l\rR_\cB(A_\gamma) \cong C^l(A_\gamma) = \left\{ s:=\sum_{i \geq 0} \lambda_{i}  u^i \in A_\gamma  ~ \Big|~ m_{A_\gamma}(s \otimes t) = m_{A_\gamma} \Psi(s \otimes  t), ~\forall  t \in A_\gamma\right\}.$$
Using the fact that $x \cdot u^i = \gamma u^{i-1}$ if $i$ is odd and $=0$ if $i$ even, we see that $m \Psi (u^i \otimes u) = -u^{i+1} + \xi \gamma^2 u^{i-1}$  if $i$ is odd and $= u^{i+1}$ if $i$ even. Therefore, from the condition defining $C^l(A_\gamma)$, we have that
$$\lambda_1 \xi \gamma^2 = 0, \quad -\lambda_i + \lambda_{i+2}\xi \gamma^2 = \lambda_i \text{ for $i$ odd}, \quad \text{ and $\lambda_i$ is free for $i$ even.}$$ 
So, $\lambda_i = 0$ for all $i$ odd, and $\lambda_i$ is free for $i$ even, no matter the choice of $\xi$ and $\gamma$. Therefore, 

\begin{proposition} The $\cB$-center of the parameterized family of $T_2(-1)$-module algebras $A_\gamma$ is 
$$Z_\cB(A_\gamma) = C^l(A_\gamma) = \Bbbk[u^2] \quad \text{with} \quad g \cdot u^2 = u^2, \quad x \cdot u^2 = 0, $$
as a commutative algebra in $\cC \cong \cB \cong \cZ_\cB(\cC)$.  \qed
\end{proposition}


\medskip

\section{On braided commutative module algebras over \texorpdfstring{$U_q(\mathfrak{g})$}{Uq(g)} and \texorpdfstring{$u_q(\mathfrak{g})$}{uq(g)}} \label{sec:uqg}

In this section, we provide an avenue to  generalize the results of Section~\ref{sec:uqsl2} on braided commutative module algebras over $u_q(\mathfrak{sl}_2)$ to those over the quantized enveloping algebras $U_q(\mathfrak{g})$ and $u_q(\mathfrak{g})$ (as in settings (\dag) and (\ddag) in the Introduction). To do so, we present the set-up of Notation~\ref{not:uqsl2} (for $u_q(\mathfrak{sl}_2)$) in the context of $U_q(\mathfrak{g})$ and $u_q(\mathfrak{g})$ in Sections~\ref{sec:Buqg}--\ref{sec:A}; here, we use the versions of quantum groups appearing in \cites{CP,BG}. We end in Section~\ref{sec:uqg-ques} with proposing several possible directions for future research to continue this work.

\subsection{The braided monoidal category \texorpdfstring{$\cB$}{B}} \label{sec:Buqg}

First, fix a Cartan datum $(I,\cdot)$. That is, let $I$ be a finite set, $L=\mZ \langle I \rangle$ be the lattice generated by $I$, and $\cdot$ a symmetric bilinear form on $L$ such that $i\cdot i$ is even and 
$a_{ij}:=2\tfrac{i\cdot j}{i\cdot i}\in \mZ_{\leq 0},$ for all $i\neq j$. Let $q$ be a free variable and $\mF:=\Bbbk(q)$. 

\smallskip

For the free abelian group $L=\mZ\langle g_i\mid i\in I\rangle$, setting
$\cal{R}(g_i,g_j)=q^{i\cdot j}$, for $i,j\in I$,
defines a dual $\cal{R}$-matrix on the group algebra 
$$\widehat{K}:=\mF L.$$ Now set $\cB$ to be the  braided monoidal category of $\widehat{K}$-comodules with the braiding obtained from $\cal{R}$.

\subsection{The setting \texorpdfstring{$(\dag)$}{(\textdagger)} for \texorpdfstring{$U_{q}(\mathfrak{g})$}{Uq(g)} in detail} \label{sec:dag}

Let $\fr{g}$ denote the semisimple Lie algebra associated  to $(I,\cdot)$, and take 
$$H:=U_q(\fr{n}^-)$$ 
to be the negative nilpotent part of the quantum group $U_q(\fr{g})$. We have that $H$ is a Hopf algebra in~$\cB$. As an algebra, $U_q(\fr{n}^-)$ is generated by primitive elements $f_i$, for $i\in I$, subject to the quantum Serre relations (as in \eqref{qgrouprel2} below for $f_i = F_i$). Further, $H$ is a $\widehat{K}$-Yetter--Drinfeld module, with $\widehat{K}$-action induced using $\cal{R}$ from the coaction $\delta$, where
\begin{align*}
g_i\cdot  f_j&=q^{-i\cdot j}f_j, & \delta(f_i)&=g_i^{-1}\otimes f_i.
\end{align*}

For $\cC=\lmod{H}(\cB)$,  we have that $\cC$ is equivalent to $\lmod{U_q(\fr{g}^-)}^{\mathrm{w}}$ and the relative monoidal center $\cZ_\cB(\cC)$ is equivalent to $\lwfmod{U_q(\fr{g})}$ as a braided monoidal category; see  \cite{L18}*{Theorem~4.7}. As in \cite{CP}*{Section~9.1}, the algebra $U_q(\mathfrak{g})$ is generated by $E_i,F_i,K_i^{\pm 1}$, for $i\in I$, subject to relations
\begin{gather}
K_iE_j=q^{i\cdot j}E_jK_i, \quad K_iF_j=q^{-i\cdot j}F_jK_i,\quad K_i^{\pm 1}K_i^{\mp 1}=1,\quad [E_i,F_j]=\delta_{i,j}\frac{K_i-K_i^{-1}}{q_i-q_i^{-1}},\label{qgrouprel1}\\
\sum_{k=0}^{1-a_{ij}}(-1)^k\binom{1-a_{ij}}{k}_{q_i}E^{1-a_{ij}-k}_i E_jE_i^k=0,\qquad \sum_{k=0}^{1-a_{ij}}(-1)^k\binom{1-a_{ij}}{k}_{q_i}F^{1-a_{ij}-k}_i F_jF_i^k=0,\label{qgrouprel2}
\end{gather}
for $i\neq j\in I$. Here, $q_i=q^{i\cdot i/2}$, and $\binom{n}{m}_q=
\frac{[n]_q!}{[m]_q![n-m]_q!}$, for $n\geq m$, where $[n]_q=\frac{q^n-q^{-n}}{q-q^{-1}}$. 
The coproduct is determined by
\begin{gather*}
\Delta(K_i)=K_i\otimes K_i, \qquad \Delta(E_i)= 1\otimes E_i + E_i\otimes K_i, \qquad \Delta(F_i)= K^{-1}_i\otimes F_i + F_i\otimes 1.
\end{gather*}

\begin{example}\label{expl:qgroup}
By Example \ref{expl:H-center},  we get that $U_q(\fr{n}^-)$ is a commutative algebra in $\lmod{U_q(\fr{g})}$ via
\begin{align*}
K_i\cdot f_j&=q^{-i\cdot j}f_j,&
F_i\cdot f_j&=f_if_j-q^{i\cdot j}f_jf_i, & E_i\cdot f_j&=\frac{\delta_{i,j}}{q_i-q_i^{-1}}.
\end{align*}
\end{example}

\subsection{The setting \texorpdfstring{$(\ddag)$}{(\textdaggerdbl)} for \texorpdfstring{$u_q(\mathfrak{g})$}{uq(g)} in detail} \label{sec:ddag}

We abuse notation and assume here that $q\in \Bbbk$  is a primitive $n$-th root of unity, where $n\geq 3$ is an odd integer and is coprime to $3$ if $\fr{g}$ contains a $G_2$-factor. Note that $q^2$ is also a primitive $n$-th root of unity. We set $q_i:=q^{i\cdot i/2}$. Now take 
$$H:=u_q(\fr{n}^-),$$ to be the negative nilpotent part of the quantum group $u_q(\fr{g})$, where $u_q(\fr{g})$ is the finite dimensional quotient of the specialization $\Ug$ by the relations $E_i^n=0$, $F_i^n=0$, $K_i^n=1$, for $i\in I$. Now $H$ is a Hopf algebra in the braided monoidal category $\cB$ (abusing notation) of $K$-modules, where $K$ is the quotient of the free group $\widehat{K}$ by the relations $g_i^n$ for $i\in I$. 

\smallskip

For $\cC=\lmod{H}(\cB)$, we have that $\cC$ is equivalent to $\lmod{u_q(\fr{g}^-)}$ and  the relative monoidal center $\cZ_\cB(\cC)$ is equivalent to $\lmod{u_q(\fr{g})}$ as a braided monoidal category; see  \cite{L18}*{Theorem~4.9}.
Moreover, an action on $u_q(\fr{n}^-)$ defined by the same formulas as in Example \ref{expl:qgroup} also makes $u_q(\fr{n}^-)$ a commutative algebra in $\lmod{\ug}$.

\subsection{Module algebras over \texorpdfstring{$u_q(\fr{g}^-)$}{uq(g-)}} \label{sec:A}

Now we show how to generalize the module algebra $A_\gamma$ over the Taft algebra $T_n(q^{-2})$ from Notation \ref{not:uqsl2} to the context of Section~\ref{sec:ddag}. Recall that $T_n(q^{-2}) \cong u_q(\fr{sl}_2^-)$ by Proposition~\ref{prop:uqaction}(2). So, here, we consider module algebras over the negative Borel part $u_q(\fr{g}^-)$ of $u_q(\fr{g})$; this is the subalgebra generated by $K_i, F_i$ for $i \in I$. 

\smallskip

Now for the braided monoidal category $\cC$ in Section~\ref{sec:ddag}, let us consider the algebra $A$ in $\cC$ defined as follows. Take $\boldsymbol{\gamma}$ to be a collection of scalars $(\gamma_i)_{i \in I}$ in $\Bbbk$, and let $A:=A_{\boldsymbol{\gamma}} = \Bbbk\langle u_i\mid i\in I\rangle$ be the free associative algebra with
 actions of $K$ and of $H=u_q(\fr{n}^-)$ given by
$$g_i\cdot u_j=q^{i\cdot j}u_j,  \quad \quad f_i\cdot u_j=\delta_{i,j}\gamma_i 1_A, \qquad \text{for }i,j\in I.$$

\subsection{Questions} \label{sec:uqg-ques}
Continuing the work in the previous section, we ask

\begin{question}
What is the $\cB$-center of the algebra $A_{\boldsymbol{\gamma}}$ in Section~\ref{sec:A}? How is it presented as a braided commutative algebra in $\lmod{u_q(\fr{g})}$?
\end{question}

In particular, one could consider the following problem.

\begin{problem}
By Example \ref{coreg-expl}, the algebra $A_{\boldsymbol{\gamma}}$, with $\gamma_i=1/(q_i-q_i^{-1})$, can be replaced with a quotient isomorphic to $u_q(\fr{n}^+)^{\Psi^{-1}}$. Compute the $\cB$-center of $u_q(\fr{n}^+)^{\Psi^{-1}}$ as the centralizer of $u_q(\fr{n}^+)^{\Psi^{-1}}$ in the braided Heisenberg double of $u_q(\fr{n}^-)$ from Definition \ref{def:braidHeis}.
\end{problem}

Moreover, motivated by Corollary~\ref{cor:MoritaSec6}, we ask:

\begin{problem}
What are conditions on the scalars $\boldsymbol{\gamma}=\{\gamma_i\}_{i \in I}$ that distinguish Morita equivalence classes for the $u_q(\mathfrak{g}^-)$-module algebras $A_{\boldsymbol{\gamma}}$?
\end{problem}

One can also consider actions on other quotients of $A_{\boldsymbol{\gamma}}$, actions in the context of Section~\ref{sec:dag} above, or revisit related work mentioned at the beginning or end of the Introduction, for further directions of investigation.


\section*{Acknowledgements} 

The authors thank Alexandru Chirvasitu, Alexei Davydov, J\"urgen Fuchs, and Christoph Schweigert for insightful discussions and references and two anonymous referees for helpful comments.
R.~Laugwitz was partially supported by an AMS-Simons travel grant. C. Walton was partially supported by a research fellowship from the Alfred P. Sloan foundation, and by the US National Science Foundation grants \#DMS-1663775, 1903192. Part of this work was carried out during a visit of R.~Laugwitz to University of Illinois, Urbana-Champaign, and hospitality of the hosting institution is gratefully acknowledged. 

\bigskip

\bibliography{bca-qea-bib}

\begin{bibdiv}
\begin{biblist}

\bib{BD}{article}{
      author={Bespalov, Yuri},
      author={Drabant, Bernhard},
       title={Hopf (bi-)modules and crossed modules in braided monoidal
  categories},
        date={1998},
     journal={J. Pure Appl. Algebra},
      volume={123},
      number={1-3},
       pages={105\ndash 129},
}

\bib{Bes}{article}{
      author={Bespalov, Yu.~N.},
       title={Crossed modules and quantum groups in braided categories},
        date={1997},
     journal={Appl. Categ. Structures},
      volume={5},
      number={2},
       pages={155\ndash 204},
}

\bib{BG}{book}{
      author={Brown, Ken~A.},
      author={Goodearl, Ken~R.},
       title={Lectures on algebraic quantum groups},
      series={Advanced Courses in Mathematics. CRM Barcelona},
   publisher={Birkh\"{a}user Verlag, Basel},
        date={2002},
}

\bib{BM}{article}{
      author={Brzezi\'nski, Tomasz},
      author={Militaru, Gigel},
       title={Bialgebroids, {$\times_A$}-bialgebras and duality},
        date={2002},
     journal={J. Algebra},
      volume={251},
      number={1},
       pages={279\ndash 294},
}

\bib{CFM}{article}{
      author={Cohen, Miriam},
      author={Fischman, Davida},
      author={Montgomery, Susan},
       title={On {Y}etter-{D}rinfeld categories and {$H$}-commutativity},
        date={1999},
     journal={Comm. Algebra},
      volume={27},
      number={3},
       pages={1321\ndash 1345},
}

\bib{Cline}{article}{
      author={Cline, Zachary},
       title={On actions of {D}rinfel'd doubles on finite dimensional
  algebras},
        date={2019},
        ISSN={0022-4049},
     journal={J. Pure Appl. Algebra},
      volume={223},
      number={8},
       pages={3635\ndash 3664},
         url={https://doi.org/10.1016/j.jpaa.2018.11.022},
}

\bib{CMZ}{article}{
      author={Caenepeel, S.},
      author={Militaru, G.},
      author={Zhu, Shenglin},
       title={Crossed modules and {D}oi-{H}opf modules},
        date={1997},
     journal={Israel J. Math.},
      volume={100},
       pages={221\ndash 247},
}

\bib{CP}{book}{
      author={Chari, Vyjayanthi},
      author={Pressley, Andrew},
       title={A guide to quantum groups},
   publisher={Cambridge University Press, Cambridge},
        date={1995},
        note={Corrected reprint of the 1994 original},
}

\bib{C-vO-Z}{inproceedings}{
      author={Caenepeel, S.},
      author={Van~Oystaeyen, F.},
      author={Zhang, Y.~H.},
       title={Quantum {Y}ang-{B}axter module algebras},
        date={1994},
   booktitle={Proceedings of {C}onference on {A}lgebraic {G}eometry and {R}ing
  {T}heory in honor of {M}ichael {A}rtin, {P}art {III} ({A}ntwerp, 1992)},
      volume={8},
       pages={231\ndash 255},
}

\bib{CW}{article}{
      author={Cohen, Miriam},
      author={Westreich, Sara},
       title={From supersymmetry to quantum commutativity},
        date={1994},
     journal={J. Algebra},
      volume={168},
      number={1},
       pages={1\ndash 27},
}

\bib{Dav1}{article}{
      author={Davydov, Alexei},
       title={Centre of an algebra},
        date={2010},
     journal={Adv. Math.},
      volume={225},
      number={1},
       pages={319\ndash 348},
}

\bib{Dav2}{article}{
      author={Davydov, Alexei},
       title={Full centre of an {$H$}-module algebra},
        date={2012},
     journal={Comm. Algebra},
      volume={40},
      number={1},
       pages={273\ndash 290},
}

\bib{DMNO}{article}{
      author={Davydov, Alexei},
      author={M\"{u}ger, Michael},
      author={Nikshych, Dmitri},
      author={Ostrik, Victor},
       title={The {W}itt group of non-degenerate braided fusion categories},
        date={2013},
        ISSN={0075-4102},
     journal={J. Reine Angew. Math.},
      volume={677},
       pages={135\ndash 177},
         url={https://doi.org/10.1515/crelle.2012.014},
}

\bib{DNO}{article}{
      author={Davydov, Alexei},
      author={Nikshych, Dmitri},
      author={Ostrik, Victor},
       title={On the structure of the {W}itt group of braided fusion
  categories},
        date={2013},
        ISSN={1022-1824},
     journal={Selecta Math. (N.S.)},
      volume={19},
      number={1},
       pages={237\ndash 269},
         url={https://doi.org/10.1007/s00029-012-0093-3},
}

\bib{Dri1}{article}{
      author={Drinfel'd, V.~G.},
       title={Quantum groups},
        date={1986},
     journal={Zap. Nauchn. Sem. Leningrad. Otdel. Mat. Inst. Steklov. (LOMI)},
      volume={155},
      number={Differentsial'naya Geometriya, Gruppy Li i Mekh. VIII},
       pages={18\ndash 49, 193},
}

\bib{Dri2}{inproceedings}{
      author={Drinfel'd, V.~G.},
       title={Quantum groups},
        date={1987},
   booktitle={Proceedings of the {I}nternational {C}ongress of
  {M}athematicians, {V}ol. 1, 2 ({B}erkeley, {C}alif., 1986)},
   publisher={Amer. Math. Soc., Providence, RI},
       pages={798\ndash 820},
}

\bib{Etingof-Gelaki}{article}{
      author={Etingof, Pavel},
      author={Gelaki, Shlomo},
       title={The small quantum group as a quantum double},
        date={2009},
     journal={J. Algebra},
      volume={322},
      number={7},
       pages={2580\ndash 2585},
}

\bib{EGNO}{book}{
      author={Etingof, Pavel},
      author={Gelaki, Shlomo},
      author={Nikshych, Dmitri},
      author={Ostrik, Victor},
       title={Tensor categories},
      series={Mathematical Surveys and Monographs},
   publisher={American Mathematical Society, Providence, RI},
        date={2015},
      volume={205},
}

\bib{FFRS}{article}{
      author={Fr\"{o}hlich, J\"{u}rg},
      author={Fuchs, J\"{u}rgen},
      author={Runkel, Ingo},
      author={Schweigert, Christoph},
       title={Correspondences of ribbon categories},
        date={2006},
     journal={Adv. Math.},
      volume={199},
      number={1},
       pages={192\ndash 329},
}

\bib{FRS}{article}{
      author={Fuchs, J\"urgen},
      author={Runkel, Ingo},
      author={Schweigert, Christoph},
       title={T{FT} construction of {RCFT} correlators. {I}. {P}artition
  functions},
        date={2002},
     journal={Nuclear Phys. B},
      volume={646},
      number={3},
       pages={353\ndash 497},
}

\bib{HK}{article}{
      author={Huang, Yi-Zhi},
      author={Kong, Liang},
       title={Open-string vertex algebras, tensor categories and operads},
        date={2004},
        ISSN={0010-3616},
     journal={Comm. Math. Phys.},
      volume={250},
      number={3},
       pages={433\ndash 471},
         url={https://doi.org/10.1007/s00220-004-1059-x},
}

\bib{HKL}{article}{
      author={Huang, Yi-Zhi},
      author={Kirillov, Alexander, Jr.},
      author={Lepowsky, James},
       title={Braided tensor categories and extensions of vertex operator
  algebras},
        date={2015},
        ISSN={0010-3616},
     journal={Comm. Math. Phys.},
      volume={337},
      number={3},
       pages={1143\ndash 1159},
         url={https://doi.org/10.1007/s00220-015-2292-1},
}

\bib{JS}{article}{
      author={Joyal, Andr\'e},
      author={Street, Ross},
       title={Tortile {Y}ang-{B}axter operators in tensor categories},
        date={1991},
     journal={J. Pure Appl. Algebra},
      volume={71},
      number={1},
       pages={43\ndash 51},
}

\bib{Kassel}{book}{
      author={Kassel, Christian},
       title={Quantum groups},
      series={Graduate Texts in Mathematics},
   publisher={Springer-Verlag, New York},
        date={1995},
      volume={155},
}

\bib{Kel}{article}{
      author={Kelly, G.~M.},
       title={Doctrinal adjunction},
        date={1974},
       pages={257\ndash 280. Lecture Notes in Math., Vol. 420},
}

\bib{KR}{article}{
      author={Kong, Liang},
      author={Runkel, Ingo},
       title={Morita classes of algebras in modular tensor categories},
        date={2008},
     journal={Adv. Math.},
      volume={219},
      number={5},
       pages={1548\ndash 1576},
}

\bib{KW}{article}{
      author={Kinser, Ryan},
      author={Walton, Chelsea},
       title={Actions of some pointed {H}opf algebras on path algebras of
  quivers},
        date={2016},
     journal={Algebra Number Theory},
      volume={10},
      number={1},
       pages={117\ndash 154},
}

\bib{L15}{article}{
      author={Laugwitz, Robert},
       title={Braided {D}rinfeld and {H}eisenberg doubles},
        date={2015},
     journal={J. Pure Appl. Algebra},
      volume={219},
      number={10},
       pages={4541\ndash 4596},
}

\bib{L17}{article}{
      author={Laugwitz, Robert},
       title={Comodule algebras and 2-cocycles over the (braided) {D}rinfeld
  double},
        date={2019},
        ISSN={0219-1997},
     journal={Commun. Contemp. Math.},
      volume={21},
      number={4},
       pages={1850045, 46},
         url={https://doi.org/10.1142/S0219199718500451},
}

\bib{L18}{article}{
      author={Laugwitz, Robert},
       title={The relative monoidal center and tensor products of monoidal
  categories},
        date={2020},
        ISSN={0219-1997},
     journal={Commun. Contemp. Math.},
       pages={1950068},
         url={https://doi.org/10.1142/S0219199719500688},
}

\bib{Lu2}{article}{
      author={Lu, Jiang-Hua},
       title={On the {D}rinfel'd double and the {H}eisenberg double of a {H}opf
  algebra},
        date={1994},
     journal={Duke Math. J.},
      volume={74},
      number={3},
       pages={763\ndash 776},
}

\bib{Lu}{article}{
      author={Lu, Jiang-Hua},
       title={Hopf algebroids and quantum groupoids},
        date={1996},
     journal={Internat. J. Math.},
      volume={7},
      number={1},
       pages={47\ndash 70},
}

\bib{Mue}{article}{
      author={M\"{u}ger, Michael},
       title={On the structure of modular categories},
        date={2003},
        ISSN={0024-6115},
     journal={Proc. London Math. Soc. (3)},
      volume={87},
      number={2},
       pages={291\ndash 308},
         url={https://doi.org/10.1112/S0024611503014187},
}

\bib{Majid}{book}{
      author={Majid, Shahn},
       title={Foundations of quantum group theory},
   publisher={Cambridge University Press, Cambridge},
        date={2000},
        ISBN={0-521-64868-8},
         url={http://dx.doi.org/10.1017/CBO9780511613104},
        note={Paperback Edition, originally published 1995},
}

\bib{Maj91}{inproceedings}{
      author={Majid, Shahn},
       title={Representations, duals and quantum doubles of monoidal
  categories},
        date={1991},
   booktitle={Proceedings of the {W}inter {S}chool on {G}eometry and {P}hysics
  ({S}rn\'\i , 1990)},
       pages={197\ndash 206},
}

\bib{Maj92}{incollection}{
      author={Majid, Shahn},
       title={Algebras and {H}opf algebras in braided categories},
        date={1994},
   booktitle={Advances in {H}opf algebras ({C}hicago, {IL}, 1992)},
      series={Lecture Notes in Pure and Appl. Math.},
      volume={158},
   publisher={Dekker, New York},
       pages={55\ndash 105},
}

\bib{Majid-cross}{article}{
      author={Majid, Shahn},
       title={Cross products by braided groups and bosonization},
        date={1994},
     journal={Journal of Algebra},
      volume={163},
      number={1},
       pages={165\ndash 190},
}

\bib{Maj99}{article}{
      author={Majid, Shahn},
       title={Double-bosonization of braided groups and the construction of
  {$U_q(\germ g)$}},
        date={1999},
     journal={Math. Proc. Cambridge Philos. Soc.},
      volume={125},
      number={1},
       pages={151\ndash 192},
}

\bib{MS}{article}{
      author={Montgomery, S.},
      author={Schneider, H.-J.},
       title={Skew derivations of finite-dimensional algebras and actions of
  the double of the {T}aft {H}opf algebra},
        date={2001},
     journal={Tsukuba J. Math.},
      volume={25},
      number={2},
       pages={337\ndash 358},
}

\bib{nlab}{misc}{
      author={{nLab authors}},
       title={doctrinal adjunction},
         how={\url{https://ncatlab.org/nlab/show/doctrinal+adjunction}},
        date={2018},
        note={Available at
  \url{https://ncatlab.org/nlab/show/doctrinal+adjunction}, Revision 16},
}

\bib{Ost}{article}{
      author={Ostrik, Victor},
       title={Module categories, weak {H}opf algebras and modular invariants},
        date={2003},
     journal={Transform. Groups},
      volume={8},
      number={2},
       pages={177\ndash 206},
}

\bib{TV}{book}{
      author={Turaev, Vladimir},
      author={Virelizier, Alexis},
       title={Monoidal categories and topological field theory},
      series={Progress in Mathematics},
   publisher={Birkh\"{a}user/Springer, Cham},
        date={2017},
      volume={322},
}

\bib{vO-Z}{article}{
      author={Van~Oystaeyen, Fred},
      author={Zhang, Yinhuo},
       title={The {B}rauer group of a braided monoidal category},
        date={1998},
     journal={J. Algebra},
      volume={202},
      number={1},
       pages={96\ndash 128},
}

\end{biblist}
\end{bibdiv}
\bibliographystyle{amsrefs}

\end{document}